\documentclass[11pt]{amsart}
\usepackage{amsmath,latexsym}

\usepackage{graphicx}
\usepackage{amssymb}
\usepackage{amscd}
\usepackage[all,curve,frame,arc,dvips]{xy}
\usepackage{pb-diagram}

\usepackage{undertilde}

\usepackage{geometry}
\usepackage{color}

\definecolor{Myblue}{rgb}{0.0,0,0.9}
\definecolor{Mygreen}{rgb}{0.2,1,0}

\newcommand{\JC}[1]{{\color{black}#1}}

\newcommand{\I}[1]{{\color{black}#1}}

\newcommand{\IP}[1]{{\color{black}#1}}
\newcommand{\MIP}[1]{{\color{black}#1}}

\newcommand{\MI}[1]{{\color{black}#1}}

\newtheorem{thm}{Theorem}[section]

\newtheorem{lem}[thm]{Lemma}
\newtheorem{prop}[thm]{Proposition}
\newtheorem{cor}[thm]{Corollary}

\theoremstyle{definition}

\theoremstyle{definition}
\newtheorem{exmp}[thm]{Example}
\newtheorem{remark}[thm]{Remark}

\numberwithin{equation}{section}

\newcommand{\La}{\Lambda}

\newcommand{\lra}{\longrightarrow}

\newcommand{\benu}{\begin{enumerate}}
\newcommand{\enu}{\end{enumerate}}

\newcommand{\bema}{\left[\begin{array}}
\newcommand{\ema}{\end{array}\right]}

\renewcommand{\mod}{\operatorname{mod}}

\newcommand{\add}{\operatorname{add}}
\newcommand{\End}{\operatorname{End}}

\newcommand{\Hom}{\operatorname{Hom}}
\newcommand{\Ext}{\operatorname{Ext}}
\newcommand{\Ker}{\operatorname{Ker}}
\newcommand{\Der}{\operatorname{D^b}}
\newcommand{\HomD}{\operatorname{Hom}_{\Der(H)}}

\newcommand{\dual}{\mathrm{D}}

\newcommand{\rad}{\operatorname{rad}}
\newcommand{\soc}{\operatorname{soc}}

\newcommand{\ind}{\operatorname{ind}}
\newcommand{\gldim}{\operatorname{gldim}}

\newcommand{\ra}{\rightarrow}

\newcommand{\HVCenter}[1]{\setbox 0=\hbox{#1}%
        \dimen0=\wd0%
        \dimen1=\ht0%
        \divide\dimen0 by 2%
        \divide\dimen1 by 2%
        \hskip -\dimen0%
        \lower \dimen1%
        \box0%
        \hskip -\dimen0}
\newcommand{\HBCenter}[1]{\setbox 0=\hbox{#1}%
        \dimen0=\wd0%
        \dimen1=\ht0%
        \divide\dimen0 by 2%
        \hskip -\dimen0%
        \box0%
        \hskip -\dimen0}
\newcommand{\LTCenter}[1]{\setbox 0=\hbox{#1}%
        \dimen1=\ht0%
        \lower \dimen1%
        \box0%
        \hskip -\dimen0}
\newcommand{\HTCenter}[1]{\setbox 0=\hbox{#1}%
        \dimen0=\wd0%
        \dimen1=\ht0%
        \divide\dimen0 by 2%
        \hskip -\dimen0%
        \lower \dimen1%
        \box0%
        \hskip -\dimen0}
\newcommand{\RTCenter}[1]{\setbox 0=\hbox{#1}%
        \dimen0=\wd0%
        \dimen1=\ht0%
        \hskip -\dimen0%
        \lower \dimen1%
        \box0%
        \hskip -\dimen0}
\newcommand{\RBCenter}[1]{\setbox 0=\hbox{#1}%
        \dimen0=\wd0%
        \dimen1=\ht0%
        \hskip -\dimen0%
        \box0%
        \hskip -\dimen0}
\newcommand{\RVCenter}[1]{\setbox 0=\hbox{#1}%
        \dimen0=\wd0%
        \dimen1=\ht0%
        \divide\dimen1 by 2%
        \hskip -\dimen0%
        \lower \dimen1%
        \box0%
        \hskip -\dimen0}
\newcommand{\LVCenter}[1]{\setbox 0=\hbox{#1}%
        \dimen1=\ht0%
        \divide\dimen1 by 2%
        \lower \dimen1%
        \box0%
        \hskip -\dimen0}

\begin{document}
\sloppy 
\title{fundamental domains  of cluster categories inside module categories}

\author[Cappa]{Juan \'Angel Cappa}
\address{Juan \'Angel Cappa, Instituto de
    Matem\'atica, Universidad Nacional del Sur, 8000, Bah\'{\i}a
    Blanca, Argentina.}
\email{juancappa2010@gmail.com}

\author[Platzeck]{Mar\'{\i}a In\'es Platzeck}
\address{Mar\'{\i}a In\'es Platzeck, Instituto de
    Matem\'atica, Universidad Nacional del Sur, 8000, Bah\'{\i}a
    Blanca, Argentina.}
\email{platzeck@uns.edu.ar}

\author[Reiten]{Idun Reiten}
\address{Idun Reiten, Institutt for Matematiske Fag,
Norges Teknisk-Naturvitenskapelige Universitet,
N-7491 Trondheim
Norway}
\email{idunr@math.ntnu.no}

\begin{abstract}

Let $H$ be a finite dimensional hereditary algebra over an algebraically
closed field, and let $\mathcal{C}_{H}$ be the corresponding cluster
category. We give a description of the (standard) fundamental domain of $\mathcal{C}_{H}
$ in the bounded derived category $\mathcal{D}^{b}(H)$, and of the
cluster-tilting objects, in terms of the category $\mod\Gamma $\ of
finitely generated modules over a suitable tilted algebra $%
\Gamma .$
Furthermore, we apply this description to obtain (the quiver of) an arbitrary
cluster-tilted algebra.

\end{abstract}

\thanks{ J.~A.~Cappa and M.~I.~Platzeck  thankfully acknowledge partial support from CONICET, Argentina and
  from Universidad Nacional del Sur, Argentina. M.~I.~Platzeck is researcher from CONICET. I.~Reiten was supported by the project 196600/v30 from the Research Council of Norway and by a Humboldt Research Prize. }


\maketitle
\section{Introduction}
Let $k$ be an algebraically closed field and $Q$ a finite acyclic quiver with n vertices. Let $H=kQ$ be the associated path algebra. The cluster category $\mathcal C_H$ was introduced and investigated in \cite{BMRRT}, motivated by the cluster algebras of Fomin-Zelevinsky \cite{FZ1}. By definition we have $\mathcal C_H = {\rm D}^b(H)/\tau^{-1}[1]$, where $\tau$ denotes the AR-translation. An important class of objects are the cluster tilting objects $T$, wich are the objects $T$  with $\Ext^1_\MI{C_H}(T,T)=0$, and $T$ maximal with this property. They are shown to be exactly the objects induced by tilting objects over some path algebra $kQ'$ derived equivalent to $kQ$.

A crucial property of the cluster tilting objects $T=T_1 \oplus \cdots \oplus T_j$ where the $T_i$ are indecomposable, and $T_i $ is not isomorphic to $ T_{i'}$ for $i\neq i'$, is that $j=n$ and for each $i=1,\cdots ,n$ there is a unique indecomposable object $T_i^* $ not isomorphic to $     T_i$ in $\mathcal C_H$, such that $(T/T_i) \oplus T_i^*$ is a cluster tilting object. This is a more regular behavior than what we have for tilting modules (of projective dimension at  most one) over a finite dimensional algebra $A$. In general there is {\it at most one}
replacement for each indecomposable summand.

The maps in $\mathcal C_H$ are defined as follows, as usual for orbit categories. Choose the fundamental domain $\mathcal D$ of $\mathcal C_H$  inside ${\rm D}^b(H)$, whose indecomposable objects are the indecomposable $H$-modules, together with \MI{$ P_1[1], \dots , P_n[1]$, where the $P_j$ are the indecomposable projective $H$-modules}. Let $X$ and $Y$ be in $\mathcal D$. Then $ \Hom_{\mathcal C_H}(X,Y) = \bigoplus_{\in Z} \Hom_{{\rm D}^b(H)}(X,\MI{(\tau^{-1}[1])}^i(Y)) $.

In \cite{ABST} the authors considered the triangular matrix algebra $
\Lambda =
\begin{pmatrix}
H & 0 \\ 
DH & H%
\end{pmatrix}
$, where ${\rm D} =\Hom_k(-,k)$. They chose a fundamental domain for ${\mathcal C_H}$ inside the category $\mod \Lambda$ of finite dimensional $\MI{\Lambda}$-modules, by using the $H$-modules together with $\I{\ind} \,\tau^{-1}_\Lambda({ D}H) $. They established a bijection between cluster tilting objects in ${\mathcal C_H}$ and a certain class of tilting modules in $\mod\Lambda$, which was shown in \cite{ABST2}   to be all tilting modules (of projective dimension at most $1$).

The present paper is inspired by \cite{ABST}. Instead of using the algebra $\Lambda$ which normally has global dimension $3$, we use \MI{a} smaller triangular matrix algebra $\Gamma$ which has global dimension at most $2$, and is a tilted algebra. We obtain a similar connection between cluster tilting objects in $\mathcal C_H$ and tilting modules in $\mod\Gamma$ and give an alternative proof for the special property of complements mentioned above.  The projective injective modules play a crucial role here, as in \cite{ABST} .

If $T$ is a tilting $H$-module, a description of the quiver of $\End_{\mathcal C_H}(T)$, on the basis of the quiver of $\End_H(T)$, is given in \cite{ABS}  (See \cite{BR} for finite type). For each \MI{relation in a minimal set of} relations in $\add T$, an arrow is added in the opposite direction. We obtain a similar description for $T$ in the fundamental domain, but not necessarily being an $H$-module. Again the projective injective modules play an essential role. Now we  consider  relations where we allow factoring also through the projective injective modules, in addition to $\add T$. Then we obtain the same result about adding arrows in the opposite direction as before. When $T$ is a tilting $H$-module, then no maps in $\add T$ factor through projective injective modules.

We now describe the content section by section. In section 2 we give some preliminary results on describing the indecomposable $\Lambda$-modules of projective dimension at most $1$. In particular, we show that all predecessors of a module of projective dimension 1 have projective dimension at most $1$. In section $3$ we introduce the algebra $\Gamma$ which replaces $\Lambda$ in our work, starting with motivation on how to choose $\Gamma$ smallest possible, without losing essential information. We show that the indecomposable $\Gamma$-modules of projective dimension at most $1$ are exactly the  modules in the left part $\mathcal L_\Gamma$ of indecomposable modules where the predecessors have projective dimension at most $1$. Further, this class consists of the indecomposable modules in our fundamental domain, together with the indecomposable projective injective $\Gamma$-modules. In section $4$ we show how to describe the quiver of $\End_{\mathcal C_H}(T)$ for any $T$ in the fundamental domain.

\section{Duplicated algebras}
{ In this section we recall work from \cite{ABST} and improve the statement of the main theorem in \cite{ABST}.} {Throughout the paper we  assume that }
 $H$ is a {basic} hereditary  algebra \MI{over an algebraically closed field $k$} and $\Lambda $ \MIP{is} the duplicated algebra of $H$, \MIP {that is,} $
\Lambda =
\begin{pmatrix}
H & 0 \\ 
DH & H%
\end{pmatrix}
$.  {We denote by $\mod \Lambda$ the category of finitely generated left $\Lambda$-modules,} and we use the usual description of the left $\Lambda -$modules as triples $%
(X,Y,f)$, with \MIP{ $X, Y$ in $\mod H$ and} $f\in $ Hom$_{H}(DH\otimes _{H}X,Y)$ \MIP{(see {\cite{FGR}, or \cite{ARS}, III,2)}. }Then the full
subcategory of $\mod\Lambda $ generated by the modules of the form $(0,Y,0)$
is closed under predecessors and canonically isomorphic to $\mod H$. We
will use this isomorphism to identify $\mod H$ with the corresponding
full subcategory of $\mod \Lambda$ and give some alternative proofs. 
\JC{The opposite algebra $\Lambda ^{op}$ is isomorphic to the triangular matrix {algebra}
$\begin{pmatrix}H^{op} & 0 \\ DH & H^{op}
\end{pmatrix}$.
Under these identifications, the duality $\dual:\mod\Lambda\lra\mod\Lambda ^{op}$ is {given by} $\dual (X,Y,f)=(\dual Y,\dual X,\dual f)$, where $\dual f\in\Hom_{H^{op}}(\dual Y,\dual(\dual H\otimes_{H} X))\cong\Hom_{H^{op}}(\dual Y,\Hom_{H}(\dual H,\dual X))$
$
\cong\Hom_{H^{op}}(\dual Y\otimes_{H}\dual H,\dual X)\cong\Hom_{H^{op}}(\dual H\otimes_{H^{op}}\dual Y,\dual X)$.}

We recall (see {\cite{FGR} or \cite{ARS}}, {III, Proposition 2.5}) that the {indecomposable} projective $\Lambda$-modules are given by triples isomorphic to those of the form $(0,P,0)$ or $(P,D H\otimes _H P,1_{\dual H\otimes _H P })$, with $P$ {indecomposable} projective in $\mod H$. The former are the projective $H$-modules, and the latter are {projective-injective $\Lambda$-modules}. The remaining {indecomposable }injective $\Lambda$-modules are of the form $(I,0,0)$ with $I$ injective in $\mod H.$

We  denote by $pd_\Lambda M$ \MIP{and $id_\Lambda M$ the projective dimension and the injective dimension of the $\Lambda$-module $M$, respectively}. When $M$ is in $\mod H$ we have $pd_H M = pd_\Lambda M$, and for that reason  we will write just $pd\, M$.  
We  denote by $\rad X$ and $\soc X$ the radical and the socle of the $\Lambda$-module $X$, respectively.

Let $\ind\Lambda$ denote the full subcategory of mod$\Lambda$ where the objects are a chosen set of nonisomorphic indecomposable $\Lambda$-modules. 

We  denote by D$^b(H)$ the bounded derived category of $H$, by $\mathcal C_H$ the cluster category of $H$, and by $\tau$ the Auslander-Reiten translation in $\mod \Lambda$ or  D$^b(H)$. Note that the injective $H-$modules are not $\Lambda -$injective, so that $\tau_\Lambda ^{-1}(I_i)$ is indecomposable for each indecomposable $H$-module $I_i = I_0(S_i)$, where $S_i$ is a simple $H$-module. Then $ \{\tau_\Lambda ^{-1}(I_i)\}$ in $\mod \Lambda$ will play a similar role as 
$\{ P_i[1]\}$ in the derived category D$^b(H)$. In particular, 
$\add (\ind H\cup\{\tau_\Lambda ^{-1}\dual H\})\subseteq\mod\Lambda$ can be considered as a fundamental domain $\mathcal D_\I{\Lambda}$ inside mod $\Lambda$ of the cluster category  $\mathcal C_H$   (see \cite{ABST}).

We recall that given  $X,Y\in \ind \Lambda$, a {\it path from $Y$ to $X$} is a sequence of nonzero morphisms  $Y\overset{f_{0}}{\longrightarrow }X_{1}$ $
\overset{f_{1}}{\longrightarrow }X_{2}\longrightarrow ...\longrightarrow
X_{t}\overset{f_{t}}{\longrightarrow }X$, with the $X_i\in \ind\Lambda$. When such a path exists,  $Y$ is a {\it predecessor} of $X$, and $X$ is a {\it successor} of $Y$. The {\it left part}  $\mathcal L_\Lambda$ of $\mod\Lambda$,  defined in \cite{HRS}, is the full subcategory of $\ind\Lambda$ consisting of the modules whose predecessors have projective dimension at most $1$. That is, 
$\mathcal L_\Lambda =\{X\in\mathrm{ind}\Lambda\ |\ {\rm pd}\,Y\leq 1\ {\rm for \ any \
predecessor}\ Y \ {\rm of}\   X\}$. 
 The main result of \cite{ABST} is the following.
\begin{thm} \label{1} (a) The fundamental domain \I{$\mathcal D_\Lambda$} of $\mathcal C_H$ lies in $\add \mathcal L_\Lambda$, and the only other indecomposable  $\Lambda$-modules  in $\mathcal L_\Lambda$ are projective-injective.

(b) There is induced a bijection between cluster-tilting objects in $\mathcal C_H$ and tilting modules in $\mod \Lambda$ whose indecomposable non projective-injective summands lie in $\mathcal L_\Lambda$.
\end{thm}

Note that in \cite{ABST2} it is shown that the bijection in (b) is with all tilting modules.   Using the following  results for  $\Lambda$-modules with projective dimension at most  one, we give a different approach to the improved version.

\begin{prop} \label{Z}

Let $X\in \ind\Lambda .$ Then $pd_\Lambda X\leq 1$ \MIP{if and only if} $\tau_\Lambda X\in 
\mod H.$ In other words, the indecomposable $\Lambda -$modules $X$ such
that $pd_\Lambda X\leq 1$ are those in  the fundamental domain of $\mathcal{C}_{H}$, together with 
the indecomposable projective-injective $\Lambda -$modules.
\end{prop}
\begin{proof} We have that $pd_\Lambda X>1$ \MIP{if and only if} Hom$_{\Lambda }(D\Lambda ,\tau X)\neq 0,$ and this last
condition implies $\tau X\notin \mod H$, since the injective
$\Lambda$-modules do not belong to $\mod H.$ Conversely, if $\tau X\notin \mod H,$ there is a projective-injective  $\Lambda -$module $E$ such that
Hom$_{\Lambda }(E,\tau X)\neq 0,$ and this implies $pd_\Lambda X>1.$ 
\end{proof}
\bigskip

\begin{lem}\label{Y}

Let $X,Y\in \ind\Lambda $ be such that Hom$_{\Lambda }(X,Y)\neq 0$ and $%
pd_\Lambda Y=1.$ Then $pd_\Lambda X\leq 1.$
\end{lem}
\begin{proof}
Let $f:X\lra Y$ be a nonzero morphism. We want to show that $pd_\La X\leq 1$. Suppose that this is not the case. Then $f$ is not an isomorphism, and so it factors through the minimal right almost split morphism $E\lra Y$. Since $f\neq 0$, we can choose an indecomposable direct summand $E_{0}$ of $E$, and morphisms $g_{0}:E_{0}\lra Y$ and $h_{0}:X\lra E_{0}$ with $g_{0} h_{0}\neq 0$ and $g_{0}$ irreducible. 
Then $E_{0}\notin\mod H$, because its predecessor $X$ is not in $\mod H$. Now suppose $E_{0}$ is projective. Then it is also injective, \I {since all indecomposable projective $\La$-modules which are not in $\mod H$ are injective.} Hence $\tau Y\cong\rad E_{0}$, because $g_{0}:E_{0}\lra Y$ is irreducible. 
Now $h_{0}:X\lra E_{0}$ factors through $\rad E_{0}\hookrightarrow E_{0}$, so rad$E_0 \notin \mod H$, and we conclude that $\tau Y\cong\rad E_{0}\notin\mod H$. By Proposition \ref{Z}, this contradicts our hypothesis $\MIP{pd\, } Y=1$. 
\MIP{Therefore $E_{0}$ is not projective, and then there is an irreducible morphism $\tau E_{0}\lra\tau Y$. On the other hand,  $pd\, Y = 1$ implies that $\tau Y $ is in $\ind H$.   Hence $\tau E_0$ is in $\ind H$, and thus $pd\, E_0 = 1$.} Therefore 
 our original  morphism  $f:X\lra Y$ can be replaced by $h_{0}:X\lra E_{0}$,  and so we can iterate the process to obtain an arbitrarily long path of irreducible morphisms $E_{m}\overset{g_{m}}{\lra}E_{m-1}\overset{g_{m-1}}{\lra}\dots\overset{g_{2}}{\lra }E_{1}\overset{g_{1}}{\lra}E_{0}$ with $g_{1}g_{2}\dots g_{m}\neq 0$.
But this is a contradiction, because each $E_{i}\in\tau^{-1}\ind H\setminus\mod H$, i.e. it is a direct summand of $\tau^{-1}\dual H$, and this implies that all the $g_{i}$ are in $\rad\End(\tau^{-1}\dual H)$, which is nilpotent.
\end{proof}

{Proposition \ref{Z} can be used to give the following relationship between $\mathcal L_\Lambda$ and the modules of projective dimension at most $1$.}
\bigskip

\begin{prop}\label{X}

\JC{If $X\in\ind\Lambda$ and $pd_\La X=1$, then} $X\in \mathcal{L}_{\Lambda }$\MIP{.}
\end{prop}
\begin{proof} Suppose \MIP{that the result does not hold}, and let $Y\overset{f_{0}}{\longrightarrow }X_{1}$ $
\overset{f_{1}}{\longrightarrow }X_{2}\longrightarrow ...\longrightarrow
X_{t}\overset{f_{t}}{\longrightarrow }X$ be a path in $\ind \Lambda $
, with $pd \, X=1$ and $pd\,Y>1.$ Then $Y\notin {\mod}H$, because $H$ is hereditary. Since ${\mod}
H$ is closed under predecessors in $\mod \Lambda ,$ then the $X_{i}$ and $X$
do not belong to ${\mod}H$ either. Now, let us choose the path so that it
has minimal length among those with $pd\, X=1$ and $pd\,Y>1.$ By Lemma \ref{Y}, $
t\geq 1.$ By minimality, $X_{i}$ is projective (and hence injective) for $
1\leq i\leq t.$ The map $f_{t-1}$ factors through ${\rm rad}X_{t}$,  which is not
injective, and thus not projective. Then, by minimality we must have $
pd({\rm rad}X_{t})>1$ and $t=1.$ Since $f_{1}$ factors through $\frac{X_{1}}{\mathrm{
soc}X_{1}}$ -which is not projective- we must have $pd\left( \frac{X_{1}}{
\mathrm{soc}X_{1}}\right) =1,$ by Lemma \ref{Y}. By Proposition \ref{Z}, $\tau \left( 
\frac{X_{1}}{\mathrm{soc}X_{1}}\right) \in \mod H.$ But then $%
1<pd\, ( {\rm rad}X_{1})=pd\left( \tau \left( \frac{X_{1}}{\mathrm{soc}X_{1}}\right)
\right) \leq 1$. \MIP{ This contradiction ends the proof of the proposition}.
\end{proof}
\bigskip

It follows that the only indecomposable $\Lambda$-modules $X$ with pd$_\Lambda X \leq 1$ and $X$ not in $\mathcal L_\Lambda$, are projective-injective. Note that we do not necessarily have that $\mathcal L_\Lambda$ consists exactly of the indecomposable $\Lambda$-modules of projective dimension at most $1$. (See Example in \cite{ABST}). We now have the following improvement of Theorem 2.1.
\begin{thm} \label{1.5} Let  $
\Lambda =
\begin{pmatrix}
H & 0 \\ 
DH & H%
\end{pmatrix}
$ as before.

(a) If $X$ is indecomposable and not projective-injective in $\mod \Lambda$, then $X$ is in $\mathcal L_\Lambda$ if and only if pd$_\Lambda X \leq 1$.

(b) The fundamental domain $\mathcal D\I{_ \Lambda}$ of $\mathcal C_H$ inside $\mod \La$ lies in $\add \mathcal L_\Lambda$, and the remaining modules in $\mathcal L_\Lambda$ are projective-injective.

(c) There is a one to one correspondence between the multiplicity-free cluster tilting objects  in $\mathcal C_H$ and the basic tilting $\Lambda$-modules.

\end{thm}

It was proven in \cite{ABST} that the global dimension gldim$\Lambda$ of $\Lambda$ is at most 3. { We end this section with a more precise description of gldim$\Lambda$. }We will give a proof of this result using the description of $\Lambda$-modules as triples, which allows us to calculate  the global dimension of $\Lambda$ more precisely. {This shows that $\Lambda$ is normally of global dimension  3.}

\begin{prop}
\label{W}   ${\rm gldim}\, \Lambda \leq 3.$ Moreover:

(a) ${\rm gldim}\, \Lambda =1$ if and only if $H$ is semisimple.

(b)  ${\rm gldim}\, \Lambda =2$ if and only if $\tau _{H}^{2}=0$ and $H$ is not semisimple.
\end{prop}
\begin{proof}
 We calculate ${\rm gldim}\, \Lambda =\max \{pd\,S:S$ simple $\Lambda -$%
module$\}$

= $\max \{pd\,(X,0,0),pd\,(0,X,0):X\in \ind H\}.$


\MI{For $M$ in $\mod H$, $  P_1(M)\ra P_0(M) \ra M \ra 0 $ denotes a minimal projective presentation.}

Let $X\in \ind H$. Then $pd\,(0,X,0)\leq 1,$ for $\mod H$ is closed
under predecessors in $\mod \Lambda .$ Suppose $X$ is  projective in $\mod H$.  Then
the following is a minimal projective resolution:
{\footnotesize $$0\longrightarrow (0,P_{1}(DH\otimes X),0)\longrightarrow (0,P_{0}(DH\otimes
X),0)\longrightarrow (X,DH\otimes X,1)\longrightarrow (X,0,0)\longrightarrow
0.$$}

Now, if $\gldim\Lambda\leq1$, then $P_{1}(\dual H\otimes X)=0$ for every projective X. Since $\dual H\otimes-$ is the Nakayama equivalence between projective and injective $H$-modules, this is to say that every injective $H$-module is also projective, i.e. $H$ is semisimple. This establishes (a), since $\Lambda$ is clearly hereditary when $H$ is semisimple.

 Assume now that  $X$ is
not projective. Then we have an exact sequence
$ 0 \rightarrow \tau X \rightarrow D(P_1(X)^*) \rightarrow D(P_0(X)^*)$. Using  that for \JC{projective $P$}  there  is a functorial isomorphism $DH \otimes P \simeq DP^*$,   we obtain the following   minimal projective resolution:

$0\longrightarrow (0,P_{1}(\tau X),0)\longrightarrow (0,P_{0}(\tau
X),0)\longrightarrow (P_{1}(X),DH\otimes P_{1}(X),1)\longrightarrow
(P_{0}(X),DH\otimes P_{0}(X),1)\longrightarrow (X,0,0)\longrightarrow 0.$ Thus pd\,$(X,0,0) \leq 2$ if and only if $P_{1}(\tau X)=0$, if and only if $\tau^2X =0$. The  proposition {now }follows  right away. 
\end{proof}
\medskip

\begin{cor} \label{V} For each algebraically closed field $k$, there are only a finite
number of basic indecomposable hereditary $k-$algebras $H$ such that  ${\rm gldim}\, \Lambda \leq 2.$ 
\end{cor}

\begin{proof} By Proposition \ref{W}, ${\rm gldim}\,  \Lambda \leq 1$ if and only if $\tau _{\Lambda
}^{2}=0,$ i.e. if each $\Lambda -$module is either projective or injective.
Hence $H$ is of finite representation type and \JC{so} its ordinary quiver $Q$ has no multiple arrows. In
addition, $(i\longrightarrow j\longrightarrow k)$ is not a subquiver of $Q$,
because the simple module $S_{j}$ would be neither projective nor injective
in such case. Finally, $(_{i}\swarrow ^{j}\searrow _{k}\swarrow ^{l})$ is
not a subquiver of $Q$, because otherwise the module {$\begin{smallmatrix} j \\ k \end{smallmatrix}$}  would be neither projective nor injective. Therefore $Q$ must be one of
the following four quivers: \JC{$\rm{A_{1}}, \rm{A_{2}},\rm{A_{3}}$ with nonlinear order $(\searrow\swarrow$ and $\swarrow \searrow)$.}\end{proof}

Denote by $\hat H = \left( \begin{smallmatrix}
 \ddots & & & & \\ & H & & &\\ 
&DH&H& &\\ 
& &DH&H &\\
& & &  &   \ddots
\end{smallmatrix} \right )$ 
the infinite dimensional repetitive algebra associated with the finite dimensional hereditary algebra $H$. As explained in \cite{ABST}, we have the following relationship:

$  \mod H \subset \mathcal D \subset \mathcal L_\La \subset \mod \La \subset \mod \hat H    \twoheadrightarrow     {\underline{\mod} \hat H} \xrightarrow{ \simeq} {\rm D}^b(H)  \twoheadrightarrow \mathcal C_H.$

The following more precise relationship will be useful.  

\begin{prop} \label{IR}  Let $\alpha: (X,Y,f) \rightarrow (X',Y',f')$  be a nonzero map in $\mod \La$. Then $\alpha $ factors through a projective injective $\La$-module if and only if it factors through a projective module in $\mod \hat H$.
\end{prop}
\begin{proof}
The projective injective $\La$-modules are additively generated by $(H, DH, id)$. For $\hat H$, the projective modules, which coincide with the injective ones,  are additively generated by modules of the form
$\left( \begin{smallmatrix}
 \ddots & & & & \\  
0&H& 0& &\\ 
 &DH& 0& &\\
& & &  &   \ddots
\end{smallmatrix} \right )$.
If $\alpha $ 
factors through a projective injective $\La$-module, it is clear that it does the same when considered as a map in $\mod \hat H$.

Conversely, assume that $\alpha$ factors through a projective $\hat H$-module.  The possible projective modules must come from one or more of the following pictures:

$(1) \left( \begin{smallmatrix}X \\Y \end{smallmatrix} \right ) \xrightarrow{\gamma }\left( \begin{smallmatrix}H \\DH \end{smallmatrix} \right ) \xrightarrow{\delta }\left( \begin{smallmatrix}X' \\Y' \end{smallmatrix} \right ) $, \, $(2) \left( \begin{smallmatrix}0\\ X \\Y \end{smallmatrix} \right ) \xrightarrow{\gamma }{\left( \begin{smallmatrix} H \\DH \\0 \end{smallmatrix} \right )} \xrightarrow{\delta }\left( \begin{smallmatrix}0\\ X' \\Y' \end{smallmatrix} \right ) $, \,
$(3) \left( \begin{smallmatrix}X \\Y \\0 \end{smallmatrix} \right ) \xrightarrow{\gamma }\left( \begin{smallmatrix}0 \\H \\DH  \end{smallmatrix} \right ) \xrightarrow{\delta }\left( \begin{smallmatrix}X' \\ Y'\\0 \end{smallmatrix} \right ) $.
In case (2)we must have a commutative diagram
$$
\CD
 {DH}
@>{}>>
{0 } @.
\\@|    @VV{}V @. \\  DH
@>{\delta}>>
X'@.\,\,\, \,\,\, ,
\endCD
$$
which is impossible since $\delta \neq 0$. In case (3) \MI{the diagram}
$$
\CD
 {DH\otimes Y}
@>{DH\otimes \gamma}>>
{DH\otimes H }
\\@VV{}V    @VV{\simeq}V  \\  0
@>{}>>
DH 
\endCD 
$$
commutes.  Then $DH \otimes \gamma$ must be zero. But since $H$ is hereditary and $\gamma:Y \ra H$ is nonzero, there is an indecomposable summand $Y_1$ of $Y$ which is projective, with $\gamma |\MI{_{Y_1}}: Y_1 \ra H$ nonzero. Since $DH\otimes -$ gives an equivalence from the category of projective $H$-modules to the category of injective ones, then $DH\otimes \gamma|Y_1$ and hence $DH\otimes \gamma$ is nonzero. This gives a contradiction.

Hence we must have case (1), which implies that $\alpha$ factors through a projective injective $\La$-module.
\end{proof}

\bigskip
We end this section with some discussion  about fundamental domains (see \cite{ABST}). For the cluster category $\mathcal C_H$ we have a natural functor $\Pi : {\rm D}^b H \ra \mathcal C_H$. Let $\mathcal D$ be the additive subcategory of D$^b(H)$ whose indecomposable objects are the indecomposable $H$-modules together with the shift $[1]$ of the indecomposable projective $H$-modules \cite{BMRRT}.
Then $\mathcal D$ is a convex subcategory of D$^b(H)$, and $\Pi$ induces a bijection between the indecomposable objects of $\mathcal D$ and those of $\mathcal C_H$. In order to find other ``fundamental domains",  one is looking for similar properties. In particular, it is nice to use appropriate module categories rather than derived categories. A step in this direction was \MI{made} in \I{\cite{ABST}}, by considering the duplicated algebra $
\Lambda =
\begin{pmatrix}
H & 0 \\ 
DH & H%
\end{pmatrix}$ of a hereditary algebra $H$. Here there is a natural functor from $\mod \La$ to $\mathcal C_H$, as discussed above, and $\mod H$ is naturally embedded into $\mod \La$. In addition, the indecomposable objects $\tau_\La^{-1}(I)$, for $I$ indecomposable injective $\MI{H}$-module, are added to $\mod \MI{H}$ to form a fundamental domain \I{ ${\mathcal D}_\Lambda$} inside $\mod \La$, giving a desired bijection with the indecomposables in $\mathcal C_H$, from our functor $\mod \La \ra \mathcal C_H$ . Here the fundamental domain is not only convex, but \MI{is also closed under  predecessors in $\mod \Lambda$ which are not  projective-injective}. We shall see that we have a similar situation when replacing the duplicated algebra $\La$ by a smaller algebra $\Gamma$.

\bigskip


\section{The algebra $\Gamma$.}

{In this section} we will replace the duplicated algebra $\Lambda $ by a smaller algebra $%
\Gamma $ such that {also} $\mod \Gamma $  contains the fundamental domain {$\mathcal D$} of $%
\mathcal{C}_{H}$.

We start with a  lemma, which will be needed later.
\begin{lem} \label{L1}
\label{U}Let $A$ be a basic artin algebra, and let $S,S^{\prime }$ be simple
projective $A-$modules. Then:

(a) If $I$ is an indecomposable injective module {not isomorphic to} 
$ I_{0}(S),$
then $\Hom _{A}(I,I_{0}(S))=0.$

(b) End$_{A}I_{0}(S)\cong $ End$_{A}S.$

(c) End$_{A}I_{0}(S\oplus S^{\prime })\cong $ End$_{A}I_{0}(S)\times $ End$
_{A}I_{0}(S^{\prime }).$

(d) $S\cong $ End$_{A}S$ as an End$_{A}S-$vector space.
\end{lem}

\begin{proof}

(a) and (b) {follow using the} Nakayama equivalence $^{\ast }D$ from the category of injective $\Lambda$-modules to the category of projective $\Lambda$-modules, and
(c) is a direct consequence of (a).

(d) Let $_{A}A=S\oplus Q.$ Then $_{\text{End}_{A}S}S\cong \Hom _{A}(A,S)=$
End$_{A}S\oplus \Hom _{A}(Q,S)=$ End$_{A}S$ , \MIP{since $\Hom_{A}(Q,S)=0$, because $A$ is
basic}. 
\end{proof}

Let $P$ be a projective $\Lambda -$module.  We recall that $\Hom _{\Lambda }(P,-):\mod 
P\longrightarrow \mod ($End$_{\Lambda }P)^{op}$ is an equivalence, where $\mod P$ is the full subcategory of $\mod \Gamma$ consisting of the modules with a presentation in $\add P$.
Now we
can take $\Gamma =($End$_{\Lambda }P)^{op},$  with the projective $P$  such that $\mod P$ contains the fundamental domain \ $\add (
\ind H\cup \{\tau ^{-1}_\La DH\})$ \ of $\mathcal C_H$. We want to choose $P$ as
small as possible. Since $\mod P\subseteq $ Gen$P,$ it is clear that add%
$P$ must contain $H\oplus P_{0}(\tau _\La ^{-1}DH).$ Next we show that this is
enough.

\medskip

We denote by $\Delta $ the sum of the nonisomorphic simple projective \MIP{$H-$modules}. That
is, $\Delta $ is a basic $\Lambda -$module such that \MIP{$\add \Delta = \add \soc
H = {\add \soc
\Lambda} $. Let $_{\Lambda }
\overline{P}=$ $H\oplus I_{0}^\Lambda (\Delta ).$ {We will prove in the next proposition that  $P_{0}(\tau_\Lambda ^{-1}DH) = I^\Lambda_{0}(H)$}, and that the basic projective module $\overline P$ has the required properties.}

\begin{prop}
\label{T} Let $_{\Lambda }\overline{P}=H\oplus I\MIP{^\Lambda _{0}}(\Delta ).$ Then:

(1) $\add \overline{P}$ is closed under predecessors in $\mathcal{P}(\Lambda
) = \{$projective $\Lambda$-modules$\}. $

(2) $\mod \overline{P}=$ Gen$\overline{P}=\{(X,Y,f)\in \mod 
\Lambda :X\in \add \Delta \}.$

(3) Let $P$ be an indecomposable projective $H-$module. Then \MIP{$\tau_{\Lambda}^{-1}\dual\Hom_{H}(P,H)=(\soc P,I_{1}^H(P),\pi)$, where $0\ra P \ra I_{0}^H (P)\xrightarrow{\pi} I_{1}^H (P)\ra 0$} is a minimal injective resolution in {$\mod H$}.

(4) $I_{0}^\Lambda  (H)=P_{0}(\tau_\Lambda ^{-1}DH).$

(5) $\mod H\cup \{\tau_\Lambda  ^{-1}DH\}\subseteq \mod \overline{P}.$


(6) If $\overline{Q}$ is a projective $\Lambda -$module such that $\mod 
H\cup \{\tau_\Lambda  ^{-1}DH\}\subseteq \mod \overline{Q},$ then $\overline{P}%
$ is a direct summand of $ \overline{Q}.$
\end{prop}
\begin{proof}

(1) 
 Let $Q\longrightarrow P$ be a nonzero morphism between indecomposable
projective $\Lambda -$modules, with $P\in \add \overline{P}.$ We have to
prove that $Q\in \add \overline{P}.$ We may assume that $P,Q$ are projective-injective. Hence $P$ is in add$I_0^\Lambda (\Delta)$, and the result follows from Lemma \ref{L1}(a).

(2) The first equality follows from (1). Now, $\overline{P}=(\Delta
,H\oplus I^H_{0}(\Delta ), \left( \begin{array}{c}0 \\ 1 \end{array} \right))
$, where we identify $DH\otimes_H \Delta$  with $I^H_0(\Delta)$. Thus $(X,Y,f)\in $ Gen$\overline{P}%
$ if and only if $ X\in $ Gen$\Delta (=\add \Delta ).$

(3) and (4). We proceed to calculate {$\tau_\Lambda ^{-1}\dual\Hom_{H}(P,H)=Tr_\Lambda\Hom_{H}(P,H)$. Let \MIP{$P^* =\Hom_{H}(P,H). $ Since $\dual\Hom_{H}(P,H)\cong (0,\dual P^{\ast},0)$  in $\mod \Lambda$ }, then $\Hom_{H}(P,H)\cong(P^{\ast},0,0)$, and the following is a minimal projective presentation:
$$(0,P_{0}(P^{\ast}\otimes_{H}\dual H),0)\longrightarrow (P^{\ast},P^{\ast}\otimes_{H}\dual H,1)\longrightarrow (P^{\ast},0,0)\longrightarrow 0.$$
Applying $\Hom_{\Lambda}(-,\Lambda)$, we obtain 
{\small {$$0\ra(0,P,0)\rightarrow ((P_{0}(P^{\ast}\otimes_{H}\dual H))^{\ast},\dual H\otimes_{H}(P_{0}(P^{\ast}\otimes_{H}\dual H))^{\ast},1)\rightarrow Tr_\Lambda\Hom_{H}(P,H)\rightarrow 0.$$}}
Since $P^{\ast}\otimes_{H}\dual H\cong\dual P$, then $(P_{0}(P^{\ast}\otimes_{H}\dual H))^{\ast}\cong(P_{0}(DP))^{\ast }=(\dual I_{0}(P))^{\ast}=(\dual I_{0}(\soc P))^{\ast}=P_{0}(\soc P)=\soc P$. (We used \MIP{ that $H$ is }hereditary in the last step). Hence $P_{0}(Tr_\Lambda \Hom_{H}(P,H))=(\soc P,I_{0}(\soc P),1)=I_{0}^{\Lambda}(\soc P)$. \MIP{
The last equality follows  from the description of injective $\Lambda$-modules as triples, and the fact that $\soc P$ is a projective $H$-module.  Therefore the above sequence is
 $$0\ra(0,P,0)\rightarrow I_0^\Lambda(\soc P)\rightarrow \tau^{-1}_\Lambda D\Hom_{H}(P,H)\rightarrow 0,$$ 
so
}  $\tau^{-1}\dual\Hom_{H}(P,H)=(\soc P,I_{1}(\MIP{ P}),\pi )$.} This establishes (3). Adding all the indecomposable projective $H-$modules \MIP{yields } { the projective resolution $$0\lra H\lra I_{0}^{\Lambda}(\soc H)\lra\tau_\Lambda^{-1}\dual H\lra 0\hskip 1in (*)$$ and proves (4).

} 
(5) \JC{We have} $\mod H\subseteq\mod\overline{P}$, {since} $H$ is a direct summand of $\overline{P}$ and $\mod H$ is closed under predecessors in $\mod\Lambda$. \JC{Now the projective resolution} \MIP{$(*)$

shows that $\tau_\Lambda ^{-1}DH\in\mod\overline{P}$.}


(6) follows from (2), (4) and (5). 
\end{proof}

\bigskip

Now we define $\Gamma =\End _{\Lambda }(\overline{P})^{op}$. {The next proposition describes $\Gamma$ as a triangular matrix ring.}
\vskip.25in

\begin{prop} 
\label{P1} Let $\Gamma =$ End$_{\Lambda }(\overline{P})^{op}$. Then:

(a) $\Gamma $ is isomorphic to the triangular matrix ring $
\begin{pmatrix}
K & 0 \\ 
J & H
\end{pmatrix}
,$ where $K=\End _{H}\Delta ^{op}$ is a basic semisimple algebra, and $
J=I_{0}^{H}(\Delta ).$ In particular, the $\Gamma -$modules can be described  in terms of triples $(_{K}X,_{H}Y,f)$, with $f:J\otimes
_{K}X\longrightarrow Y.$ 

(b) For $X\in \add\Delta $ there is an isomorphism $J\otimes _{K}$Hom$
_{H}(\Delta ,X)\overset{\psi }{\simeq }DH\otimes _{H}X$ of $H-$modules which
is functorial in $X.$

(c) $(_{H}X, _HY,f)\longmapsto (_{K}$Hom$_{H}(\Delta ,X),_{H}Y,f\psi )$ is
an equivalence from $\mod \overline{P}$ to $\mod \Gamma .$
\end{prop}  

\begin{proof}

(a) Since $_{\Lambda }\overline{P}=$ $H\oplus I_{0}^{\Lambda }(\Delta ),$
 then
 $$\Gamma =\End_{\Lambda }(\overline{P})^{op}\simeq 
\begin{pmatrix}
\text{End}_{\Lambda }H & \text{Hom}_{\Lambda }(I_{0}^{\Lambda }(\Delta ),H)
\\ 
\text{Hom}_{\Lambda }(H,I_{0}^{\Lambda }(\Delta )) & \text{End}_{\Lambda
}I_{0}^{\Lambda }(\Delta )
\end{pmatrix}
^{op}  $$

$\simeq 
\begin{pmatrix}
H^{op} & 0 \\ 
\text{Hom}_{\Lambda }(H,I_{0}^{\Lambda }(\Delta )) & \text{End}_{\Lambda
}I_{0}^{\Lambda }(\Delta )
\end{pmatrix}
^{op}\simeq 
\begin{pmatrix}
\text{End}_{\Lambda }I_{0}^{\Lambda }(\Delta )^{op} & 0 \\ 
\text{Hom}_{\Lambda }(H,I_{0}^{\Lambda }(\Delta )) & H
\end{pmatrix}
.$
\bigskip

Now, by Lemma \ref{L1}, End$_{\Lambda }I_{0}^{\Lambda }(\Delta )\simeq $ End$
_{\Lambda }\Delta \simeq $ End$_{H}\Delta \simeq $ $K^{op}$ is a basic
semisimple algebra.

Finally, since $I_{0}^{\Lambda }(\Delta )=(\Delta ,I_{0}^{H}(\Delta ),1),$
we have Hom$_{\Lambda }(H,I_{0}^{\Lambda }(\Delta ))\simeq $ Hom$%
_{H}(H,I_{0}^{H}(\Delta ))\simeq I_{0}^{H}(\Delta )=J$ as  $H-K$ bimodule.

(b) Since the functors $J\otimes _{K}$Hom$_{H}(\Delta ,-)$ and $DH\otimes
_{H}-$ are additive, we can assume that $X$ is simple projective. By Lemma
\ref{L1}(c), we have $K\simeq \prod\limits_{S\in \text{ind}\Delta }$End$
_{H}(S)^{op},$ so that $J\otimes _{K}$Hom$_{H}(\Delta ,X)\simeq
I_{0}^{H}(X)\otimes _{\text{End}_{H}(X)^{op}}$End$_{H}(X)\simeq $ End$
_{H}(X)\otimes _{\text{End}_{H}(X)}I_{0}^{H}(X)\simeq I_{0}^{H}(X).$ But $
DH\otimes _{H}X$ is also isomorphic to $I_{0}^{H}(X)$ when $X$ is simple
projective.

(c) Let $(_{H}X,_{H}Y,f)\in \mod \overline{P}.$ By Proposition \ref{T}(2),
we have that $X$ is a semisimple projective $H-$module. Now the statement
follows easily from (b). 
\end{proof}
\vskip .17in

{Note that the equivalence  given in  Proposition \ref{P1}(c) is just $ \Hom_\Lambda(\overline P,-): \mod \overline P \rightarrow \mod \Gamma$, stated in terms of triples.} We will identify $\mod \Gamma$ with the full subcategory mod$\overline P$ of mod$\Lambda$. Under this identification, the fundamental domain { $\mathcal D\I{_\Lambda}$ of $\mathcal C_H$} in $\mod \La$  is in $\mod\Gamma$, $_{\Lambda}\Gamma$ = $_\Lambda\overline P$, and $I_0^\Gamma (H)=I_0^\Lambda (H)=P_0^\Lambda (\tau_\Lambda^{-1}\dual H)=P_0^\Gamma (\tau_\Gamma^{-1}\dual H)$. From Proposition \ref{T}{(2), it} follows easily that a minimal $\Lambda$-projective resolution of a $\Gamma$-module $M$ is in $\mod\Gamma$. Hence also $pd_\Gamma M=pd_\Lambda M$ for M in mod$\Gamma$.

\vskip .25 in
\MIP{We illustrate the situation with the following example.}
\begin{exmp}

 For the hereditary algebra $H$ given below we indicate the corresponding algebras $\Lambda$ and   $\Gamma$.
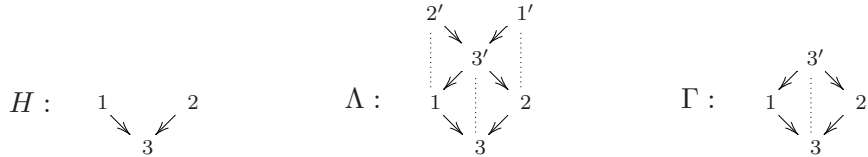
\begin{align*}
&\begin{xy}<0.60cm,0cm>:
0*+[o]{\begin{smallmatrix}\ 3\end{smallmatrix}}="a"; 
(1,1)*+[o]{\begin{smallmatrix}\ 2\end{smallmatrix}}="e";
(-1,1)*+[o]{\begin{smallmatrix}\ 1\end{smallmatrix}}="f";
(-2.5,1)*+[o]{ H :}
\ar@{->}  "f";"a"
\ar@{->}  "e";"a"
\end{xy}
&&   
\begin{xy}<0.60cm,0cm>:
0*+[o]{\begin{smallmatrix}\ 3\end{smallmatrix}}="a"; 
(1,1)*+[o]{\begin{smallmatrix}\ 2\end{smallmatrix}}="e";
(-1,1)*+[o]{\begin{smallmatrix}\ 1\end{smallmatrix}}="f";
(0,2)*+[o]{\begin{smallmatrix}\ 3'\end{smallmatrix}}="g";
(-1,3)*+[o]{\begin{smallmatrix}\ 2'\end{smallmatrix}}="h";
(1,3)*+[o]{\begin{smallmatrix}\ 1'\end{smallmatrix}}="j";
(-2.5,1.1)*+[o]{  \Lambda :}
\ar@{->} "h"; "g"
\ar@{->} "j"; "g"
\ar@{->} "g"; "f"
\ar@{->} "g"; "e"
\ar@{->}  "f";"a"
\ar@{->}  "e";"a"
\ar@{.}  "g";"a"
\ar@{.}  "j";"e"
\ar@{.}  "h";"f"
\end{xy}
&&   
\begin{xy}<0.60cm,0cm>:
0*+[o]{\begin{smallmatrix}\ 3\end{smallmatrix}}="a"; 
(1,1)*+[o]{\begin{smallmatrix}\ 2\end{smallmatrix}}="e";
(-1,1)*+[o]{\begin{smallmatrix}\ 1\end{smallmatrix}}="f";
(0,2)*+[o]{\begin{smallmatrix}\ 3'\end{smallmatrix}}="g";
(-2.5,1.1)*+[o]{  \Gamma :}
\ar@{->} "g"; "f"
\ar@{->} "g"; "e"
\ar@{->}  "f";"a"
\ar@{->}  "e";"a"
\ar@{.}  "g";"a"
\end{xy}
\end{align*}

\end{exmp}

\medskip

{By using $\Gamma$ instead of $\Lambda$ we get improved versions of Propositions \ref{X} and \ref{W}.}

\medskip

\begin{prop}\label{E}

(a) $\Gamma $ is a tilted algebra.

(b) The {set of }indecomposable $\Gamma -$modules with projective dimension $\leq 1$
is closed under predecessors, and consists of the {indecomposable objects in  the} fundamental domain $\mathcal D\I{_\Gamma}$  of $
\mathcal{C}_{H}$ plus the indecomposable projective-injective $\Gamma -$
modules. {In particular, $ \mathcal L_\Gamma$ consists of the indecomposable $\Gamma$-modules of projective dimension at most one.}
\end{prop}

\begin{proof}

(a)
Let $U =DH\oplus I_{0}^\Gamma(\Delta ).$  By \cite{APT}, Lemma 2.1, it suffices to show that $U$ is a convex tilting $\Gamma$-module.

We have $pd _\Gamma U=pd_H(DH)\leq 1$  because $I_{0}^\Gamma(\Delta )$ is projective and $H$ is hereditary.
 Since  $\mod H$ is closed under predecessors, then Ext$_{\Gamma }^{1}(U,U)=$ Ext$_{\Gamma }^{1}(DH,DH)=$ Ext$
_{H}^{1}(DH,DH)=0$. 
Finally, $\left\vert \ind \add U\right\vert =$ rk K$_{0}(\Gamma )$, where K$_{0}(\Gamma )$ denotes the Grothendieck group of $\Gamma$.
Hence $U$ is tilting. 

Now let us see that $U$ is convex:

Let $T_{0}\overset{f_{1}}{\longrightarrow }T_{1}\overset{f_{1}}{%
\longrightarrow }...\overset{f_{s}}{\longrightarrow }T_{s}$ be a path in ind$%
\Gamma $ with $T_{0},T_{s}\in \add U$, where we assume that all $f_i$ are \MI{non-isomorphisms}. If $T_{s}\in \mod H,$ then all $%
T_{i}$ are $H-$modules, and therefore they are all $H-$injective. On the other hand, if $T_{s}\notin  \mod H,$ then $%
T_{s}\in \add I_{0}(\Delta ).$ Then $T_{s}$ is projective and $f_{s}$
factors through $\rad T_{s}$. Since $\rad T_{s}\in \mod H$ and is an injective $H$-module, we are in the
previous case, so we are done.

(b) Let $f:X\rightarrow Y$ be a non-zero map with $X,Y\in \ind \Gamma $
and $pd\,Y\leq 1.$ If $pd\, Y=1$ then, by Proposition \ref{X}, $pd\, X\leq 1.$ Thus we can
assume $Y$ is projective and $f$ is not surjective. Hence $f$ factors
through $\rad Y.$ Since $\rad Y\in \mod  H$, the proposition follows. 
\end{proof}
\medskip
\begin{cor}\label{2.5}
Let $\mathcal F = \{X\in {\rm \mod}\Gamma: pdX \leq 1\}$, $\mathcal T = \add (\ind \Gamma  \setminus
 \mathcal F) = \add \{ X\in \ind\Gamma : pdX =2\}$. Then $(\mathcal T, \mathcal F)$ is a {split} torsion pair in $\mod \Gamma$.
\end{cor}

\medskip Now we will prove that the global dimension of End$_{\Gamma }(T)$
remains less than or equal to $2$ when $T$ is a tilting $\Gamma -$module. We
will use results from \cite{GHPRU}, which we collect in
the following lemma.
\vskip .15in

\begin{lem}\label{Gastaminza}
Let $A$ be an artin algebra with finite global dimension and $
_{A}T$ be a tilting module such that $pd\, T=1$ and $id\,T=s.$ Let $B= (\End_{A}T)^{op}$.
Then:

\rm{(a)} (\cite{GHPRU} Prop. 2.1) We have  $s\leq $ gldim$\,B\leq s+1.$ 

\rm{(b)} (\cite{GHPRU} Thm. 3.2) If $s\geq 1,$ then $s=$ gldim$\,B$ if and only if
$\Ext_{A}^{s}(\tau T,T)=0. $

\end{lem}
\medskip
\begin{prop}

Let $T$ be a tilting $\Gamma -$module. Then ${\rm gldim}\,  {\rm End}_\Gamma(T)\leq 2.$
\end{prop}
\begin{proof}
  By Proposition \ref{E}(a), $\Gamma $ is a tilted algebra. Thus gldim$\,
\Gamma \leq 2.$ Then we may assume that $pd\, T=1.$ Let $s= id\,T.$ If $s\leq
1$, the proposition follows from Lemma \ref{Gastaminza}(a), so assume $s=2.$ By Lemma
\ref{Gastaminza}(b), it is enough to prove that Ext$_{\Gamma }^{2}(\tau T,T)=0.$ We have
Ext$_{\Gamma }^{2}(\tau T,T)\simeq $ Ext$_{\Gamma }^{1}(\Omega \tau
T,T)\simeq $ $\mathrm{D}$Hom$_{\Gamma }(T,\tau \Omega \tau T).$ Therefore it
suffices to show that $\tau \Omega \tau T=0.$ By Proposition \ref{E}(b), pd$\tau
T\leq 1.$ Hence $\Omega \tau T$ is projective and $\tau \Omega \tau T=0$. 
\end{proof}

{We have seen that} the algebra $\Gamma$ is a tilted algebra and ${\mod}\, \Gamma$ contains the fundamental domain $\mathcal D\I{_\Gamma}$ of the cluster category  $\mathcal C_H$ as a {full subcategory}, closed under predecessors in $\mod \Gamma$.  An analogous statement to Theorem \ref{1.5}(c) also holds for tilting $\Gamma$-modules. We will use a preliminary lemma.

\begin{lem} \label{Juan} 

There is a bijective correspondence between the set of (isoclasses of) basic tilting $\Gamma$-modules and the set of (isoclasses of) basic tilting $\Lambda$-modules, given by \hskip .5in
$ _\Gamma T \longrightarrow _\Lambda T \oplus I_0^\Lambda (DH)/I_0^\Lambda (\Delta) $.

\end{lem} 
\begin{proof}
Let $T$ be a basic tilting $\Gamma $-module and let $\overline Q=I_0^\Lambda (DH)/I_0^\Lambda (\Delta) $. Then $_\Lambda \Lambda \simeq \overline P \oplus \overline Q$. Hence the basic projective-injective $\Lambda$-module $\overline Q$ is not in mod$\Gamma$, and $T \oplus \overline Q$ is a basic partial tilting $\Lambda$-module. But $|$ind($T \oplus \overline Q)|= |\ind T| + |\ind \overline Q|= |\ind \Gamma| + |\ind \Lambda| - |\ind \overline P| = |\ind \Lambda|$. Thus $T \oplus \overline Q$  is a basic tilting $\Lambda$-module. Conversely, let $T'$ be a basic tilting $\Lambda$-module. Then the basic projective-injective $\Lambda$-module $\overline Q$ is a direct summand of $T'$. Now, since $pd\, T' \leq 1$, by Proposition \ref{Z}, we have that $\tau_\Lambda T'$ is in mod$H$. Thus $T'$ is in add$(\tau^{-1}_\Lambda \mod H\cup \{_\Lambda \Lambda\}) \subseteq \add (\mod\Gamma \cup \{\overline Q \})$, and therefore $T'/\overline Q$ is a basic partial tilting $\Gamma$-module, which must be a tilting $\Gamma$-module, by the  counting argument used before.
\end{proof}

\begin{thm} \label {14} There is a bijective correspondence $\theta$ between the {multiplicity-free} cluster-tilting objects in the cluster category $C_H$ of $H$ and the   {basic}   tilting $\Gamma$-modules. For a cluster-tilting object represented by $T$ in the fundamental domain $\mathcal D\I{_\Gamma}$ of $C_H$, the corresponding tilting $\Gamma$-module is 
$\theta (T) = T \oplus  I^\Gamma_0(\Delta)$.
\end{thm}
\begin{proof}
Let $T\in \mathcal D\I{_\Gamma}$ be a multiplicity-free cluster-tilting object in $C_H$. By \cite{ABST} Thm. 10, $T \oplus  I^\Lambda_0(DH)$ is a basic tilting $\Lambda$-module.
Thus, by  Lemma \ref{Juan}, $\theta (T) = T \oplus  I^\Gamma_0(\Delta)=  T \oplus  I^\Lambda_0(\Delta)$ is a basic tilting $\Gamma$-module. Conversely, if $T'$ is a basic tilting $\Gamma$-module then, by Lemma \ref{Juan}, $ T '\oplus I_0^\Lambda (DH)/I_0^\Lambda (\Delta) $ is a basic tilting $\Lambda$-module. Hence, by Theorem \ref{1.5} (c), $T'/I_0^\Gamma (\Delta)= T'/I_0^\Lambda (\Delta) $ is in the fundamental domain $\mathcal D\I{_\Lambda } $, and represents a multiplicity-free cluster-tilting object in $C_H$.
\end{proof} 

As a consequence of this result we obtain the following result of \cite{BMRRT}.

\begin{cor} Let $H$ be a hereditary algebra. Then each almost complete cluster tilting object in $C_H$ has exactly two complements.
\end{cor}
\begin{proof}
Let $T'$ be an almost complete cluster tilting object in $C_H$. Then $T' \oplus  I_0^\Gamma(\Delta)$ is an almost complete tilting module in $\mod \Gamma$. Since  $ \add \Delta = \add (\soc \Gamma )$, then  $I^\Gamma_0(\Delta)$ is a  {faithful $\Gamma$-module, since $\Gamma \subseteq I_0^\Gamma(\Delta)$.} We know from a result of Happel and Unger that then $T' \oplus  I^\Gamma_0(\Delta)$ has exactly two complements, thus so does $T'$ in $C_H$ (see \cite{HU}).
\end{proof} 
 The following result, building upon Proposition \ref{IR}, will be useful in the next section. Here $\utilde{\mathcal D}_\La$ and $\utilde{\mathcal D}_\Gamma$ denote the categories $\mathcal D_\La$ and $\mathcal D_\Gamma$
 modulo the projective injective $\La$-modules, respectively the projective injective $\Gamma$-modules.
 \begin{prop}\label{IR2} (a) A map $\alpha: (X,Y,f) \rightarrow (X',Y',f')$ in $\mathcal D_\La \subset \mod \La$  factors through a projective injective $\La$-module if and only if it factors through $I_0^\La(\soc H)$.

(b)We have an embedding of $\utilde{\mathcal D}_\Gamma$ into D$^b(H)$ via the composition
$\utilde{\mathcal D}_\Gamma = \utilde{\mathcal D}_\La \subset \utilde{\mod} \La \subset \utilde{\mod} \hat H = \underline{\mod} \hat H \xrightarrow{\simeq} {\rm D} ^b(H)$.
\end{prop}
\begin{proof}
The indecomposable objects in $\mathcal D_\La$  are the indecomposable $H$-modules together with the $\tau^{-1}_\La(I)$ for $I$ an indecomposable injective $H$-module.
Similarly, the indecomposable objects in $\mathcal D_\Gamma$  are the indecomposable $H$-modules together with the $\tau^{-1}_\Gamma(I)$ for $I$ an indecomposable injective $H$-module. If $f:X\ra Z$ is an irreducible map between indecomposable modules in $\mod \La$, with $Z$ projective injective and $X$ in $\mod H$, it is clear that $Z= \left( \begin{smallmatrix}S \\I_0(S) \end{smallmatrix} \right ) $, where $S$ is simple projective, since all proper predecessors of $Z$ should be $H$-modules. Since then only summands of $I_0^\La(\soc H)$ are amongst the projective injective $\La$-modules which are summands of the middle term of an almost split sequence $0\ra I \ra E \ra \tau^{-1}_\La I \ra 0$ \, in $\mod \La$, we see that $\tau_\Gamma^{-1} I =\tau^{-1}_\La I$, so that $\mathcal D_\Gamma = \mathcal D_\La $.

It follows from the above that the only indecomposable projective injective $\La$-modules which have a nonzero map to $\tau^{-1}_\La I $ are  the summands of $I_0^\La(\soc H)$. Hence no nonzero map $g: D \ra D'$ in $\mathcal D_\La$ can factor through any other projective injective $\La$-modules. (Note, however,  that there might be additional projective injective $\La$-modules belonging to $\mathcal L_\La$).  Now we conclude that $\utilde{\mathcal D}_\Gamma = \utilde{\mathcal D}_\La$, and the rest follows.

\end{proof}

\section{a description of the cluster tilted algebras.}

In this section our aim is to describe the quivers of  cluster tilted algebras, that is, of the endomorphism algebras of {cluster-tilting objects} in $\mathcal C_H$,  using the fundamental domain $\mathcal D_\Gamma$ for $\mathcal C_H$  inside mod$\Gamma$. Note that a cluster tilted algebra is determined by its quiver \cite{BIRS}. Let $\hat T$ be a cluster-tilting object in $\mathcal C_H$. We assume that $\hat T $ is represented by $T$ in $\mathcal D_\Gamma \subset \mod \Gamma$. For $X, Y$ in $\mathcal D_\Gamma$, regarded as objects in D$^b(H)$,  we have  that $\Hom_{\mathcal C_H}(X,Y) = \HomD(X,Y) \oplus \HomD(F^{-1}X,Y) $, where $F=\tau^{-1} [1]$ in D$^b(H)$ \MI{(see \cite{BIRS})}. By Proposition \ref{IR2}      we have  $ \HomD(X,Y) = \utilde{\Hom}_\Gamma(X,Y)$.
We want to investigate how to describe $\HomD(F^{-1}X,Y) $ in terms of $\mod \Gamma$, for $X, Y \in \add T$.

\MIP{We first assume  that $T$ is an $H$-module. In this case  
$\HomD(F^{-1}T,T) \simeq \Ext^2_B(DB,B)$, where $B=\End_{D^b(H)}(T)$ (\cite{ABS}, proof of Theorem 2.3). The top of this $B$-$B$-bimodule  is generated as $k$-vector space by a minimal set of  relations of $B$ (\cite{ABS}, 2.2 and 2.4).} These relations correspond to  relations  between projective $B$-modules. Since projective $B$-modules are of the form $\Hom_{\Der(H)} (T,T') = \Hom_\MI{H}(T,T')$, such relations correspond to  relations between indecomposable  modules in $\add T$. 

By this we mean the following. Let $T= T_1\oplus \dots \oplus T_n$ with the $T_i$ indecomposable  and pairwise non-isomorphic. Let now $i\neq j$. We will consider maps $f: T_i \rightarrow T_j$  which are irreducible in $\add T$, that is, the maps  which do not factor through a module in $\add(T/(T_i\oplus T_j))$. Let $A(i,j)$ be the space $\Hom_H(T_i,T_j)$  modulo the maps  which factor through  $\add(T/(T_i\oplus T_j))$.  For each pair  $(i, j) $ with $ i \neq j$, choose a set of irreducible maps in $\add T$ representing  a basis of $A(i,j)$, and let $\mathcal B$ be the union of all these bases. To each path of maps in $\mathcal B$ we associate the corresponding composition map in mod$H$. A linear combination of such paths is a {\it relation for $\add T$ } if the corresponding map is zero in mod$H$.
 A set $R$ of such relations is a {\it minimal set of relations for $\add T$} if
$R$ is a minimal set of generators of the ideal of relations for $\add T$. This means that 
for any relation $g:T_r \ra T_s$ we have $g= \sum a_i \gamma_i \rho_i \gamma'_i$, with $a_i$ in $ k, \rho_i$ in R, and $\gamma_i, \gamma'_i$ paths in add$T$; 
and no proper subset of $R$ has this property.

We will prove that a similar statement holds also when the $\Gamma$-module  $T$  is not  an $H$-module. In this case we have to consider a minimal set of relations   {between indecomposable summands of $T$ in  $ \add (T \oplus I_0^\Gamma(\Delta))$.}

\vskip .1 in
Consider the following example.
Let $H=kQ$, where $Q$ is the quiver
$\begin{smallmatrix}1\end{smallmatrix} \rightarrow\begin{smallmatrix}2\end{smallmatrix}$. Then $\Gamma$ is the path algebra of the quiver $\begin{smallmatrix}2'\end{smallmatrix} \rightarrow \begin{smallmatrix}1\end{smallmatrix} \rightarrow \begin{smallmatrix}2\end{smallmatrix}.$ Let $\mathcal D$ be the fundamental domain of $\mathcal C_H$.
\begin{align*}
&\begin{xy}<0.60cm,0cm>:
0*+[o]{\begin{smallmatrix}\ 1\end{smallmatrix}}="a"; 
(-2,0)*+[o]{\begin{smallmatrix}\ 2\end{smallmatrix}}="b";
(1,1)*+[o]{\begin{smallmatrix}\ 2[1]\end{smallmatrix}}="e";
(-1,1)*+[o]{\begin{smallmatrix}\ 1\\ \ 2\end{smallmatrix}}="f";
(2,0)*+[o]{\begin{smallmatrix}\ 1\\ \ 2 \end{smallmatrix}[1]}="c";
(-3.5,1)*+[o]{ \mathcal D :}
\ar@{->}  "b";"f"
\ar@{->}  "f";"a"
\ar@{->}  "a";"e"
\ar@{->}  "e";"c"
\end{xy}
&&   
\begin{xy}<0.60cm,0cm>:
0*+[o]{\begin{smallmatrix}\ 1\end{smallmatrix}}="a"; 
(-2,0)*+[o]{\begin{smallmatrix}\ 2\end{smallmatrix}}="i";
(2,0)*+[o]{\begin{smallmatrix}\ 2'\end{smallmatrix}}="h";
(1,1)*+[o]{\begin{smallmatrix}\ \, 2'\\  \ 1\end{smallmatrix}}="e";
(-1,1)*+[o]{\begin{smallmatrix}\ 1\\ \ 2\end{smallmatrix}}="f";
(0,2)*+[o]{\begin{smallmatrix}\ 2'\\ \, 1\\ \ 2\end{smallmatrix}}="g";
(-5.5,1.1)*+[o]{ \text{AR quiver of } \Gamma :}
\ar@{->} "f"; "g"
\ar@{->} "g"; "e"
\ar@{->}  "f";"a"
\ar@{->}  "a";"e"
\ar@{->}  "i";"f"
\ar@{->}  "e";"h"
\end{xy}
\end{align*}

and let $T_1= 2, \ \ T_2 = \begin{smallmatrix} 1\\ 2 \end{smallmatrix}[1]$, and $T=T_1\oplus T_2$. Then $T$ defines a cluster tilting object in $\mathcal C_H$ and is not an $H$-module. The $\Gamma $-module corresponding to $T$ under the identification of $\mathcal D$ with $\utilde{\mod }\, \Gamma$ is $2 \oplus 2'$. Moreover,  $\Hom_{\mathcal C_H}(T_2, T_1) = \Hom_{D^bH}(\tau \, T_2[-1], T_1)= \Hom_{D^bH}(\tau\begin{smallmatrix}\ 1\\ \ 2\end{smallmatrix}, T_1) \neq 0$, but there are no relations  in add$T$ from $T_1$ to $T_2$.  However \MI {$ \begin{smallmatrix}\ 2  \end{smallmatrix} \ra \begin{smallmatrix}\ 2'\\ \, 1\\ \ 2\end{smallmatrix}  \ra \begin{smallmatrix}\ 2'\end{smallmatrix}$  is} a zero relation from $T_1$ to $T_2$ in $\add(T\oplus I)$, where $I =\begin{smallmatrix}\ 2'\\ \ 1\\ \ 2\end{smallmatrix}$ is the injective envelope of the simple $2 $ in mod\,$\Gamma$.
\vskip .1in

 To study the general case we will define an appropriate hereditary algebra $\tilde H$, and use that the above mentioned result  of \cite{ABS} holds for tilting \MIP{$\tilde H$-modules},   to prove our desired result. 

We start {with} defining a hereditary algebra $\tilde H$ such that there is an exact embedding $ G:\mod\Gamma \rightarrow \mod(\tilde H)$ with the property  that \IP{tilting $\Gamma$-modules} map to tilting $\tilde H$-modules.

\vskip .15in

We recall from Proposition \ref{E} that $U =DH \oplus I_0^\Gamma(\Delta)$ is a complete slice in $\mod \Gamma $. We consider another complete slice,
$\Sigma =\tau_\Gamma ^{-1}\mathrm{D}H\oplus I_0^\Gamma(\Delta).$ Then $\tilde H= (\End_\Gamma(\Sigma ))^{op}$ is a hereditary algebra of type $\Sigma $.  \MIP {Let  $(\mathcal T, \mathcal F)$ be the split torsion pair in $\mod \Gamma$ of Corollary \ref{2.5}. Then }$\ind \mathcal F$ coincides with the predecessors of $\Sigma $. Also $D\Sigma $ is a tilting $\tilde H$-module,  $\Gamma=\End _{\tilde H}(D\Sigma )^{op}$ and the functors $L = \Hom_{\tilde H}(D\Sigma , -) $ and $ \Ext^1_{\tilde H}(D\Sigma, -): \mod {\tilde H} \rightarrow \mod \Gamma$ induce equivalences $\mathcal T_{D\Sigma}  \rightarrow \mathcal F$ and $\mathcal F_{D\Sigma} \rightarrow \mathcal T$, respectively, where  $(\mathcal T _{D\Sigma} , \mathcal F_{D\Sigma})$ is the torsion pair associated to the tilting ${\tilde H}$-module $D\Sigma $.

Let $G= D\Sigma \otimes_\Gamma - : \mod \Gamma \rightarrow \mod \tilde H$. Then $L$ and $G$ are adjoint functors, and the restrictions  $L|_{\mathcal T_{D\Sigma}} : \mathcal T_{D\Sigma} \rightarrow \mathcal F$  and $G|_{\mathcal F}: \mathcal F \rightarrow \mathcal T_{D\Sigma} $ are inverse equivalences of categories.
Moreover, $G(\Sigma)=D\tilde H$, because $D\tilde H \simeq GLD\tilde H = G \Hom_{\tilde H}(D\Sigma , D\tilde H) \simeq G \Hom_{\tilde H^{op}}(\tilde H , \Sigma)\simeq G(\Sigma). $

\begin{exmp}\label{E2}
We illustrate the situation for the hereditary algebra $H$ indicated below.
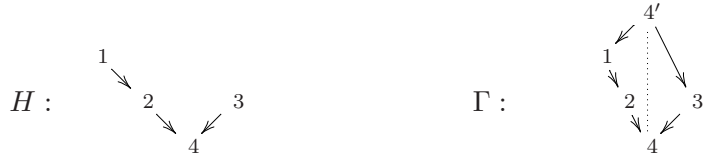
\begin{align*}
&\begin{xy}<0.60cm,0cm>:
0*+[o]{\begin{smallmatrix}\ 4\end{smallmatrix}}="a"; 
(1,1)*+[o]{\begin{smallmatrix}\ 3\end{smallmatrix}}="e";
(-1,1)*+[o]{\begin{smallmatrix}\ 2\end{smallmatrix}}="f";
(-2,2)*+[o]{\begin{smallmatrix}\ 1\end{smallmatrix}}="g";
(-3.5,1)*+{H:}
\ar@{->}  "f";"a"
\ar@{->}  "e";"a"
\ar@{->}  "g";"f"
\end{xy}
&&   
\begin{xy}<0.60cm,0cm>:
0*+[o]{\begin{smallmatrix}\ 4\end{smallmatrix}}="a"; 
(1,1)*+[o]{\begin{smallmatrix}\ 3\end{smallmatrix}}="e";
(-.5,1)*+[o]{\begin{smallmatrix}\ 2\end{smallmatrix}}="h";
(-1,2)*+[o]{\begin{smallmatrix}\ 1\end{smallmatrix}}="f";
(0,3)*+[o]{\begin{smallmatrix}\ 4'\end{smallmatrix}}="g";
(-3.5,1)*+{\Gamma:}
\ar@{->} "g"; "f"
\ar@{->} "f"; "h"
\ar@{->} "g"; "e"
\ar@{->}  "h";"a"
\ar@{->}  "e";"a"
\ar@{.} "g"; "a"
\end{xy}
\end{align*}
\medskip

The Auslander Reiten quiver of $\Gamma $ is:

\medskip
{\footnotesize
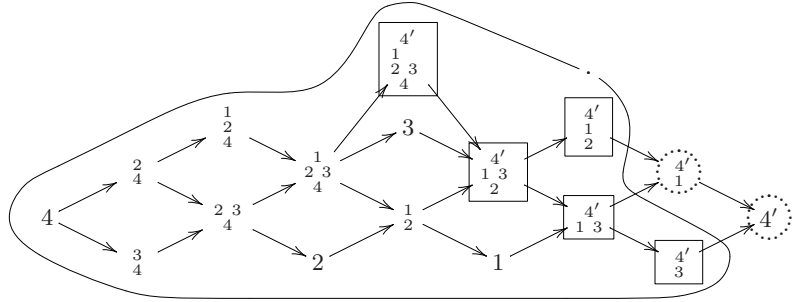
{$$
\begin{xy}<0.60cm,0cm>:
0*+[o]{\begin{smallmatrix}\ 2\ \\4\end{smallmatrix}}="a";
(-2,-1)*+[o]{4}="w"; (0,-2)*+[o]{\begin{smallmatrix}\ 3\ \\4\end{smallmatrix}}="x";(4,-2)*+[o]{2}="y";(8,-2)*+[o]{1}="z";(12,-2)*+[o]{\begin{smallmatrix}\ 4'\\3\end{smallmatrix}}="zz"*\frm{-};(14,-1)*+[o]{4'}="xx";
(2,1)*+[o]{\begin{smallmatrix}\ 1\  \\2 \\4\end{smallmatrix}}="d";
(2,-1)*+[o]{\begin{smallmatrix}\ 2\ 3\ \\4\end{smallmatrix}}="e";
(6,2.5)*+[o]{\begin{smallmatrix}4' \\ 1\ \ \  \\2\ 3\ \\4\ \end{smallmatrix}}="f"*\frm{-};
(4,0)*+[o]{\begin{smallmatrix}\ 1\  \\2\ 3 \\4 \end{smallmatrix}}="g";
(6,1)*+[o]{3}="i";
(6,-1)*+[o]{\begin{smallmatrix}1\\2\end{smallmatrix}}="j";
(8,0)*+[o]{\begin{smallmatrix}4'\\1\ 3\ \\2\ \end{smallmatrix}}="l"*\frm{-};
(10,1)*+[o]{\begin{smallmatrix}\ 4'  \\1 \\2\end{smallmatrix}}="n"*\frm{-};
(10,-1)*+[o]{\begin{smallmatrix}\ 4'\\1\ 3\end{smallmatrix}}="p"*\frm{-};
(12,0)*+[o]{\begin{smallmatrix}\ 4'\\1\end{smallmatrix}}="r";
 (16,-2)*+[o]{.}="v",
 (10,2.2)*+[o]{.}="p1";(10,2.2)*+[o]{.}="p2",**\crv{(6.5,3.7)&(6.1,3.8)&(5.5,3.7)&(5.2,3.5)&(4.2,1.5)&(2,2)&(0,1)&(-2,0)&(-2.5,-.2)&(-3,-1)&(-2.5,-1.8)&(0,-2.8)&(1,-2.8)&(11,-2.8)&(12,-2.8)&(13,-2.5)&(13.2,-1.8)&(13,-1.5)&(11,- .5)&(10.6,-0.2)&(11.2,1.3)&(10.3,2.1)},
 "r",*\xycircle<8pt>{{.}},
 "xx",*\xycircle<8pt>{{.}},
 \ar@{->}  "a";"d"
 \ar@{->}  "a";"e"
\ar@{->}  "e";"g"
\ar@{->}  "x";"e"
\ar@{->}  "g";"f"
\ar@{->}  "d";"g"
\ar@{->}  "f";"l"
 \ar@{->}  "zz";"xx"
\ar@{->}  "z";"p"
 \ar@{->}  "p";"zz"
\ar@{->}  "g";"i"
\ar@{->}  "g";"j"
 \ar@{->}  "j";"z"
 \ar@{->}  "y";"j"
\ar@{->}  "i";"l"
\ar@{->}  "j";"l"
 \ar@{->}  "w";"x"
 \ar@{->}  "e";"y"
\ar@{->}  "l";"n"
\ar@{->}  "l";"p"
\ar@{->}  "w";"a"
\ar@{->}  "n";"r"
\ar@{->}  "p";"r"
 \ar@{->}  "r";"xx"
\end{xy}
$$ }}
\bigskip 

Here, \ $_\Gamma \Sigma$ is the sum of the five modules in  frames, ${\mathcal T}$ is given by the two modules inside dotted circles, ${\mathcal F}$ is indicated by the curve, and
$\tilde H$ is the algebra with quiver:
\bigskip
{\footnotesize
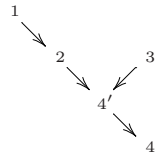
{$$
\begin{xy}<0.60cm,0cm>:
0*+[o]{\begin{smallmatrix}\ 4'\end{smallmatrix}}="a"; 
(1,1)*+[o]{\begin{smallmatrix}\ 3\end{smallmatrix}}="e";
(-1,1)*+[o]{\begin{smallmatrix}\ 2\end{smallmatrix}}="f";
(-2,2)*+[o]{\begin{smallmatrix}\ 1\end{smallmatrix}}="g";
(1,-1)*+[o]{\begin{smallmatrix}\ 4\end{smallmatrix}}="h";

\ar@{->}   "a";"h"
\ar@{->}  "f";"a"
\ar@{->}  "e";"a"
\ar@{->}  "g";"f"
\end{xy}
$$
}}
\bigskip The Auslander Reiten quiver of $\tilde H$ is
\bigskip
{\footnotesize
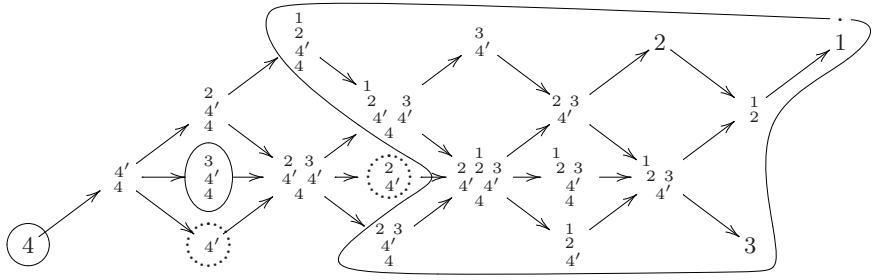
{$$
\begin{xy}<0.60cm,0cm>:
0*+[o]{\begin{smallmatrix}2 \\ \ 4' \\ 4\end{smallmatrix}}="a";
(-2,-1.5)*+[o]{\begin{smallmatrix}\ 4' \\4\end{smallmatrix}}="w"; (0,-3)*+[o]{\begin{smallmatrix}\ 4'\end{smallmatrix}}="x";(4,-3)*+[o]{\begin{smallmatrix}\ 2\ 3\ \\4'\\4\end{smallmatrix}}="y";(8,-3)*+[o]{\begin{smallmatrix} 1\\ 2 \\ \ 4'\end{smallmatrix}}="z";(12,-3)*+[o]{3}="zz";(14,1.5)*+[o]{1}="xx";
(2,1.5)*+[o]{\begin{smallmatrix}1 \\2 \\ \ 4' \\ 4\end{smallmatrix}}="d";
(2,-1.5)*+[o]{\begin{smallmatrix}\ 2\ \, 3\ \\ \ 4'\ 4' \\4\end{smallmatrix}}="e";
(0,-1.5)*+[o]{\begin{smallmatrix}3 \\ \ 4' \\ \ 4\  
\end{smallmatrix}}="f";
(4,-1.5)*+[o]{\begin{smallmatrix}2 \\ \ \ 4'\  
\end{smallmatrix}}="ff";
(8,-1.5)*+[o]{\begin{smallmatrix}1\, \, \, \, \, \, \, \,\\ 2\ 3 \\ \ 4' \\ \ 4\  
\end{smallmatrix}}="fff";
(4,0)*+[o]{\begin{smallmatrix}\ 1\, \, \, \, \, \, \, \, \, \, \, \, \,  \  \\2\, \, \, \, \, \ 3 \\ \ 4'\ 4' \\  4 \end{smallmatrix}}="g";
(6,1.5)*+[o]{\begin{smallmatrix}3 \\ \  4'\end{smallmatrix}}="i";(-4,-3)*+[o]{4}="aa";
(6,-1.5)*+[o]{\begin{smallmatrix}1\\2\ 2\ 3\\ 4'\ 4'\\4\end{smallmatrix}}="j";
(8,0)*+[o]{\begin{smallmatrix}2\ 3\ \\4'\ \end{smallmatrix}}="l";
(10,1.5)*+[o]{2}="n";
(10,-1.5)*+[o]{\begin{smallmatrix}1 \, \, \, \, \, \, \, \, \\ 2\ 3 \\ \ 4'\end{smallmatrix}}="p";
(16,-3)*+[o]{.}="v";
(12,0)*+[o]{\begin{smallmatrix}\ 1\\ \ 2\end{smallmatrix}}="r";
"ff",*\xycircle<8pt>{{.}},
"x",*\xycircle<8pt>{{.}},
"aa",*\xycircle<8pt>{},
"f",*\xycircle<9pt,13pt>{},
(14,2)*+[o]{.}="q1";(14,2)*+[o]{.}="q1",**\crv{(2,2.5)&(1,2.3)&(1.5,1.0)&(3.5,-0.5)&(4.43,-1.07)&(5.2,-1.5)&(4.43,-1.93)&(3.57,-2.57)&(2.43,-3.5)&(4,-3.7)&(12,-3.6)&(12.6,-3.6)&(12.6,-3)&(12,0)&(14.43,1.07)&(15,2)},
\ar@{->}  "a";"d"
\ar@{->}  "a";"e"
\ar@{->}  "e";"g"
\ar@{->}  "x";"e"
 \ar@{->}  "ff";"j"
\ar@{->}  "d";"g"
 \ar@{->}  "e";"ff"
 \ar@{->}  "f";"e"
  \ar@{->}  "j";"fff"
 \ar@{->}  "r";"xx"
\ar@{->}  "z";"p"
 \ar@{->}  "p";"zz"
\ar@{->}  "g";"i"
\ar@{->}  "g";"j"
 \ar@{->}  "j";"z"
 \ar@{->}  "y";"j"
\ar@{->}  "i";"l"
\ar@{->}  "j";"l"
 \ar@{->}  "w";"x"
\ar@{->}  "fff";"p"
 \ar@{->}  "e";"y"
 \ar@{->}  "l";"n"
\ar@{->}  "l";"p"
\ar@{->}  "aa";"w"
 \ar@{->}  "w";"f"
\ar@{->}  "w";"a"
\ar@{->}  "n";"r"
\ar@{->}  "p";"r"
\end{xy}
$$ }}
\bigskip 

In this case \,  $ _HD\Sigma =  4\oplus {\begin{smallmatrix}1 \\2 \\ \ 4' \\ 4\end{smallmatrix}} \oplus {\begin{smallmatrix}3 \\ \ 4' \\ \ 4\  
\end{smallmatrix}}\oplus  {\begin{smallmatrix} 2\ 3 \\ \ 4' \\ \ 4\  
\end{smallmatrix}}\oplus {\begin{smallmatrix}1\, \, \,\, \, \, \, \, \ \\ 2\  3 \\ \ 4' \\ \ 4\  
\end{smallmatrix}} \, $,\, \, 
${\mathcal F}_
{D\Sigma} $ consists of the two modules inside dotted circles,
 and $\mathcal T_{D\Sigma}$ consists of the $15$ modules inside the regions indicated with curves.

\end{exmp}

\vskip .2in We know  from the Brenner-Butler theorem that $\mathcal F = \Ker \rm{ Tor}^\Gamma _1(D\Sigma ,-)$, so that $G= D\Sigma \otimes_\Gamma  - : \mod \Gamma \rightarrow \mod \tilde H$ restricted to $\mathcal F$ is exact. In particular, $G| _{\mod H}$ is exact. Thus $G$ induces an embedding $$\hat G : \Der(\mod H) \rightarrow \Der (\mod \tilde H)$$
such that $\hat G(M[i]) = (\hat G(M))[i].$ 

\begin{prop}\label{16} Let $T$ be a tilting $\Gamma$-module. Then $G(T)$ is a tilting $\tilde H$-module.
\end{prop}
\begin{proof}
Since $\Gamma$ and $\tilde H$ have the same number of nonisomorphic simple modules, we only need to prove that $\Ext^1_{\tilde H}(G(T), G(T)) = 0$. Since $T \in \mathcal F$ because pd$T \leq 1$, it follows that $G(T) \in \mathcal T_{D\Sigma }$. Then $\Ext^1_\Gamma(LG(T), LG(T))=\Ext^1_{\tilde H}(G(T), G(T))$.  Thus $\Ext^1_{\tilde H}(G(T), G(T)) \simeq \Ext^1_\Gamma(T, T) = 0$, since $LG(T) \simeq T$ because $T \in \mathcal F$.
\end{proof}

We observe that  in general   the modules  $\tau(G(X))$ and $ G(\tau X)$ are not isomorphic, for $X \in \mod \Gamma$. We will prove next that they are isomorphic when $X=\tau_\Gamma^{-1}I$, for any injective $H$-module $I$. We start with \IP{three} lemmas.
\vskip .15in
\begin{lem}\label{B} Let $f:X\longrightarrow Y$ be a morphism in $\mod 
\Gamma ,$ with $Y\in \mathcal{F}$, such that $G(f):G
X\longrightarrow G Y$ is (minimal) right almost split in $\mod 
\tilde {H}.$ Then $f$ is (minimal) right almost split in $\mod
\Gamma .$
\end{lem}
\begin{proof}

 The key is that $\mathcal{F}$ is closed under
predecessors and $G|_{\mathcal F}: \mathcal F \rightarrow \mathcal T_{D\Sigma } $ is an equivalence. From this we obtain right away that $X\in 
\mathcal{F}$,  and $f$ is not a split epimorphism since $G(f)$ is not a split epimorphism. Now, let $Z$ be an indecomposable $
\Gamma -$module and $h:Z\longrightarrow Y$ a morphism which is not a split epimorphism. Again, $
Z\in \mathcal{F}$ and $G( h)$ is not an isomorphism. Hence there is
a morphism $g:G(Z)\longrightarrow G(X)$ such that $G(h)=G(f) g
.$ Using that \IP{$G|\mathcal F : \mathcal F \ra \mathcal T_{D\Sigma}$} is an equivalence once more, we deduce there
is a $g^{\prime }:Z\longrightarrow X$ such that $G( g^{\prime })=g$ and $
h=fg^{\prime }.$ Thus $f$ is  right almost split . The
minimality is deduced in the same way. 
\end{proof}
\bigskip 

Let again $\Sigma $ denote  the complete slice in $\mod \Gamma $ which consists of 
$\ind  (\tau_\Gamma ^{-1}\mathrm{D}H\oplus I^\Gamma_{0}(\mathrm{soc}H)).$

\begin{lem} The indecomposable projective-injective modules (i.e. those in $
\ind (I_{0}^\Gamma(\mathrm{soc}H))$) are sources of $\Sigma .$
\end{lem}
\begin{proof} Let $P\in $ $\mathrm{ind}(I_{0}(\mathrm{soc}H))$ and let $
f:X\longrightarrow P$ be a nonzero non-isomorphism in  $\mathrm{ind}\Gamma .$
Then $\rm{Im} f\subseteq \mathrm{rad}P$, which is an injective $H-$module.
Thus $X\in \mod H,$ since $\mod H$ is closed under predecessors in $
\mod\Gamma .$ But then $X\notin \Sigma .$ 
\end{proof}
\bigskip 

\begin{lem}\label{lema A} Let $I$ be an indecomposable injective $\tilde{H}-$module, and let $M\longrightarrow I$ be a {minimal }right almost split morphism in $\mod \tilde H$. Then $
M\in \mathcal{T}_{D\Sigma }.$
\end{lem}
\begin{proof} Let $M^{\prime }$ be an indecomposable direct summand of $
M$, and assume $M^{\prime }\notin \mathcal{T}_{D\Sigma } .$ Then $M^{\prime }$ is not
injective and there is an irreducible morphism $I\longrightarrow \tau
^{-1}M^{\prime }.$ Since $\tilde{H}$ is hereditary, then $\tau
^{-1}M^{\prime }$ is injective. Thus $\tau ^{-1}M^{\prime }\in \mathcal T_{D\Sigma }$, and since $M'\notin \mathcal T_{D\Sigma }$ we conclude that  $\tau ^{-1}M^{\prime }$ is
Ext-projective in $\mathcal{T}_{D\Sigma }.$ Therefore, there exists $N\in \ind
\Gamma $ such that $N$ is Ext-projective in $\mathcal{F}$ and $G
N= \tau ^{-1}M^{\prime }.$ Since $\mathcal{F}$ is closed under predecessors, then $N$
is projective in $\mod\Gamma .$ But we also have $N\in \add \Sigma $,
because $G $ maps $\Sigma $ to $\mathrm{D}\tilde{H},$ \IP{as we observed before Example \ref{E2}}. Since $\tau
^{-1}\mathrm{D}H$ contains no nonzero projective direct summands, then $N\in \mathrm{
ind}(I_{0}(\mathrm{soc}H)),$ i.e. $N$ is projective-injective. By the preceding lemma, $N
$ is a source in $\Sigma .$ Thus  $\tau ^{-1}M^{\prime }=G N$ is a source in $%
\mathrm{ind}(\mathrm{D}\tilde{H})$, which contradicts the already
established existence of an irreducible morphism $I\longrightarrow \tau
^{-1}M^{\prime }.$
\end{proof}
\bigskip 

\begin{prop}\label{C} Let $I$ be an indecomposable injective $H-$module.
Then $G \tau_\Gamma ^{-1}I=\tau _{\tilde H}^{-1}G I.$
\end{prop}
\begin{proof} Since $\tau_\Gamma^{-1}I\in \Sigma $, then $G \tau_\Gamma ^{-1}I$
is $\tilde{H}-$injective. Let $f:M\longrightarrow G \tau_\Gamma ^{-1}I$ be
a minimal right almost split morphism. By Lemma \ref{lema A}, {we have} $M\in \mathcal{T}_{D\Sigma }.$ Then
there is a morphism $g:N\longrightarrow \tau_\Gamma ^{-1}I$ in $\mod \Gamma$ with $N\in \mathcal{F}$%
, such that $f=G g.$ By Lemma \ref{B}, $g$ is minimal right almost split in $\mod\Gamma .$ Then the almost split sequence $0\longrightarrow
I\longrightarrow N\overset{g}{\longrightarrow }\tau_\Gamma ^{-1}I\longrightarrow 0$
is contained in $\mathcal{F}$, and applying $G $ we obtain an exact
sequence $(\ast )\qquad 0\longrightarrow G I\longrightarrow M\overset{f}%
{\longrightarrow }G \tau_\Gamma ^{-1}I\longrightarrow 0.$

Since $f$ is minimal right almost split, then the sequence $(\ast )$ is
almost split. Hence $G I=\tau_{\tilde H} G\tau_\Gamma ^{-1}I,$ and the result
follows { by } applying $\tau_{\tilde H} ^{-1}$. 
\end{proof}

\begin{prop}
Let $\hat T$ be a cluster tilting object in $C_H$ represented by $T$ in the fundamental domain $\mathcal D_\Gamma$ of $C_H$, which we consider embedded in $\mod \Gamma$ { as before}. Let $T_1$, $T_2 $ be \IP{indecomposable summands of $T$.} Then {\rm top} $ 
\HomD(F^{-1}T_1,T_2) $ is a vector space with basis given by a minimal set of relations from $T_2$ to $T_1$  in  $ \add (T \oplus I_0^\Gamma(\Delta))$.
\end{prop}
\begin{proof}
As we observed above, the result holds for summands $T_1, T_2$ of $T$ which are $H$-modules. So we only need to consider the case when $T_1 \notin \mod H$, that is, $T_1 = \tau^{-1}_\Gamma I$, {where} $I$ is an indecomposable injective module in $\mod H$. For if $T_1 \in \mod  H$ and $T_2 \notin \mod H$, we have $\Hom(\tau T_1[-1],T_2)=0$ since $T_2=P_i[1]$ for $P_i$ indecomposable projective \cite{ABST}.

Then $\tau^{-1}_\Gamma I = \tau^{-1}_{\Der(H)} I = P[1]$ in $\Der(H)$, where top$P= $ soc$I${, via our identification}. Then, for $X\in \mod H$ we have $$\HomD (F^{-1}(\tau^{-1}I),X) = \HomD ( I[-1],X) \simeq \Hom_{\Der(\tilde H)} ( \hat G (I[-1]),\hat GX)= $$  $$\Hom_{\Der(\tilde H)} ( (G I)[-1],GX)= \Hom_{\Der(\tilde H)} ( F^{-1} \tau^{-1}(G I), GX).$$
From Proposition \ref{C} we know that $G \tau ^{-1}I=\tau ^{-1}G I.$  Thus $$\HomD (F^{-1}(\tau^{-1}I),X) \simeq \Hom_{\Der(\tilde H)} (( F^{-1} G (\tau^{-1} I)),  GX).$$

Using Proposition \ref{C} again we can prove that this isomorphism induces an isomorphism between the corresponding tops.

Recall that $T\oplus I^\Gamma_0(\Delta)$ is a tilting $\Gamma$-module (Theorem  \ref{14}) and therefore $G(T\oplus I^\Gamma_0(\Delta) )$ is a tilting $\tilde H$-module (Proposition \ref{16}) . 

Now consider  $X=T_2$. Since both $G (\tau_\Gamma^{-1} I)=G(T_1)  $ and $  G(T_2)$ are modules over the hereditary algebra $\tilde H$,  we can apply \cite{ABST} to the tilting module  $G(T\oplus I^\Gamma_0(\Delta) )$  and conclude that
 top\,$\Hom_{\Der(\tilde H)} (( F^{-1} G (T_1),  GT_2)$ has a basis in correspondence with a minimal set of   relations from $G(T_2) $ to  $G(T_1)$in $\add G(T\oplus I^\Gamma_0(\Delta))$. 
Since $G|_{\add (T\oplus I^\Gamma_0(\Delta))}$ is an equivalence of categories, we obtain minimal relations as stated.
\end{proof}

Let $T$ and $\hat T$ be as in the previous proposition. We are now in {the} position to describe the ordinary quiver $Q_C$ of the cluster-tilted algebra $C= \End_{\mathcal C_H}(\hat T)$, in terms of $\mod \Gamma$. 

\begin{thm}
Let $C = \End_{\mathcal C_H }(\hat T)$, where $\hat T$ is a basic cluster-tilting object in $C_H$ represented by $T= \bigoplus T_i$ in $\mod\Gamma$, with $T_i$ indecomposable. Let $B= \MI{{\utilde\End}_\Gamma} ( T)$, and let  $i$ denote the vertex of $Q_C$ associated to $\MI{\utilde{\Hom}}_\Gamma(T,T_i)$. Then, for vertices $i, j$ of $Q_C$ the number of arrows from $i$ to $j$ is equal to the number of arrows from $i$ to $j$ in $Q_B$ plus the {cardinality} of a minimal set of 
  relations from $T_i$ to $T_j$  in  $ \add (T \oplus I_0^\Gamma(\Delta))\subset \mod \Gamma$.

\begin{proof}
The number of arrows from $i$ to $j$ equals dim top $\Hom_{\mathcal C_H}(T_j,T_i)= {\rm dim \, top} \HomD(T_j,T_i) \oplus {\rm dim \, top} \HomD(F^{-1}T_j,T_i).$ Now the result follows {from }the previous proposition and the fact that ${\rm dim \, top} \HomD(T_j,T_i)$ is equal to the number of arrows from $i$ to $j$ in $Q_B$, \MI{because $\HomD(T_j,T_i) \simeq \utilde{\Hom}_\Gamma (T_j,T_i)$, by Proposition \ref{IR2} (b).}
\end{proof}

\begin{remark}
In the above statement, for each pair of vertices $i$ and $j$, only one of the summands describing the number of arrows from $i$ to $j$ is nonzero. 

\end{remark}

\end{thm}


\begin{exmp}

For the hereditary algebra $H$ given below we indicate the corresponding algebra  $\Gamma$.
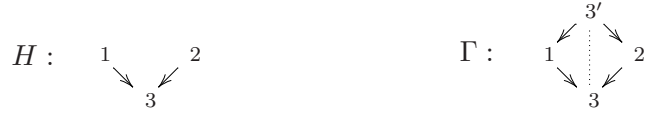
\begin{align*}
&\begin{xy}<0.60cm,0cm>:
0*+[o]{\begin{smallmatrix}\ 3\end{smallmatrix}}="a"; 
(1,1)*+[o]{\begin{smallmatrix}\ 2\end{smallmatrix}}="e";
(-1,1)*+[o]{\begin{smallmatrix}\ 1\end{smallmatrix}}="f";
(-2.5,1)*+[o]{ H :}
\ar@{->}  "f";"a"
\ar@{->}  "e";"a"
\end{xy}
&&   
\begin{xy}<0.60cm,0cm>:
0*+[o]{\begin{smallmatrix}\ 3\end{smallmatrix}}="a"; 
(1,1)*+[o]{\begin{smallmatrix}\ 2\end{smallmatrix}}="e";
(-1,1)*+[o]{\begin{smallmatrix}\ 1\end{smallmatrix}}="f";
(0,2)*+[o]{\begin{smallmatrix}\ 3'\end{smallmatrix}}="g";
(-2.5,1.1)*+[o]{  \Gamma :}
\ar@{->} "g"; "f"
\ar@{->} "g"; "e"
\ar@{->}  "f";"a"
\ar@{->}  "e";"a"
\ar@{.}  "g";"a"
\end{xy}
\end{align*}
\vskip .1in
Then the AR-quiver of $\Gamma$ is:
{\footnotesize
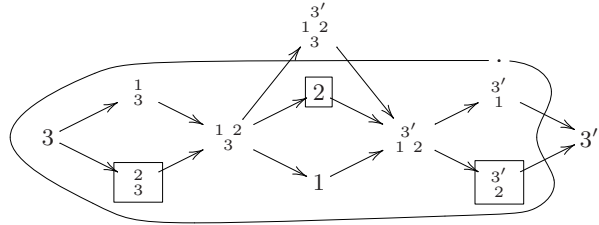
{$$
\begin{xy}<0.60cm,0cm>:
0*+[o]{3}="a";
(2,1)*+[o]{\begin{smallmatrix}\ 1\ \\3\end{smallmatrix}}="d";
(2,-1)*+[o]{\begin{smallmatrix}\ 2\ \\3\end{smallmatrix}}="e"*\frm{-};
(6,2.5)*+[o]{\begin{smallmatrix}3'\\1\ 2\ \\3\ \end{smallmatrix}}="f";
(4,0)*+[o]{\begin{smallmatrix}\ 1\ 2\ \\3 \end{smallmatrix}}="g";
(6,1)*+[o]{2}="i"*\frm{-};
(6,-1)*+[o]{1}="j";
(8,0)*+[o]{\begin{smallmatrix}3'\\\ 1\ 2\ \end{smallmatrix}}="l";
(10,1)*+[o]{\begin{smallmatrix}3'\\\ 1\ \end{smallmatrix}}="n";
(10,-1)*+[o]{\begin{smallmatrix}3'\\\  2\ \end{smallmatrix}}="p"*\frm{-};
(12,0)*+[o]{3'}="r";
 (13,-2)*+[o]{.}="v",
(10,1.7)*+[o]{.}="p1";(10,1.7)*+[o]{.}="p2",**\crv{(2,1.7)&(1.5,1.7)&(-2,0)&(1.5,-1.7)&(2,-2)&(10,-1.7)&(10.8,-1.7)&(11.5,-1)&(10.5,0)&(11.5,1)&(10.8,1.7)},

\ar@{->}  "a";"d"
\ar@{->}  "a";"e"
\ar@{->}  "e";"g"
\ar@{->}  "g";"f"
\ar@{->}  "d";"g"
\ar@{->}  "f";"l"
\ar@{->}  "g";"i"
\ar@{->}  "g";"j"
\ar@{->}  "i";"l"
\ar@{->}  "j";"l"
\ar@{->}  "l";"n"
\ar@{->}  "l";"p"
\ar@{->}  "n";"r"
\ar@{->}  "p";"r"
\end{xy}
$$ }}
\noindent  
and the fundamental domain of $\mathcal C_H$ corresponds to the region enclosed  by the curve. Let $T =\begin{smallmatrix}2\\3\end{smallmatrix}\oplus 2 \oplus \begin{smallmatrix}3'\\2\end{smallmatrix}$. Then $T\oplus {\begin{smallmatrix}3'\\1\ 2\ \\3\ \end{smallmatrix}} $ is a tilting $\Gamma$-module, so $T$ defines a  cluster tilting object $\overline T$ in $\mathcal C_H$.

We notice that nonzero maps
$\begin{smallmatrix}2\\3\end{smallmatrix} \rightarrow \begin{smallmatrix}3'\\1\ 2\\3\end{smallmatrix} \rightarrow \begin{smallmatrix}3'\\2\end{smallmatrix}$
, or
$\begin{smallmatrix}2\\3\end{smallmatrix} \rightarrow  2 \rightarrow \begin{smallmatrix}3'\\2\end{smallmatrix}$ have always nonzero composition.
However, there are nonzero maps 
$\begin{smallmatrix}2\\3\end{smallmatrix} \rightarrow \begin{smallmatrix}3'\\1\ 2\\3\end{smallmatrix} \oplus 2 \rightarrow \begin{smallmatrix}3'\\2\end{smallmatrix}$ 
with zero composition, and this relation from  $\begin{smallmatrix}2\\3\end{smallmatrix}$  to $\begin{smallmatrix}3'\\2\end{smallmatrix}$ is unique, up to scalar multiples. Therefore $\dim \Hom_{\mathcal C_H}(\begin{smallmatrix}3'\\2\end{smallmatrix},\begin{smallmatrix}2\\3\end{smallmatrix}) =1$.
\vskip .1in

Since $\dim \Hom_\Gamma(\begin{smallmatrix}2\\3\end{smallmatrix},2) =1$, and $\dim \Hom_\Gamma(2,\begin{smallmatrix}3'\\2\end{smallmatrix}) =1$, the ordinary quiver of the cluster tilted algebra $\End_{\mathcal C_H}(\overline T)$  is
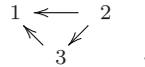
$$\begin{xy}<0.60cm,0cm>:
0*+[o]{\begin{smallmatrix}\ 3\end{smallmatrix}}="g"; 
(1,1)*+[o]{\begin{smallmatrix}\ 2\end{smallmatrix}}="e";
(-1,1)*+[o]{\begin{smallmatrix}\ 1\end{smallmatrix}}="f";
(2,0)*+[o]{.}
\ar@{->} "g"; "f"
\ar@{->} "e"; "g"
\ar@{->}  "e";"f"
\end{xy}$$

\end{exmp}
 \bigskip

{\bf Acknowledgement.} This work started during a visit of the third author    in Bah\'ia Blanca in 2005, and she would like to thank her coauthors and the other members of the algebra group for their hospitality. The second author visited Trondheim in 2009 and she would like to thank the hospitality of the members of the representation theory group.


\medskip
\begin{thebibliography}{99}
 

\bibitem{ABS}{I.~Assem, Th.~Br\"ustle, R.~Schiffler: \sl{Cluster
   tilted algebras as trivial extensions}. Bull. Lond. Math. Soc. 40 (2008), no. 1, 151--162.}


\bibitem{ABST}{I.~Assem, Th.~Br\"ustle, R.~Schiffler and G.~Todorov: \sl{Cluster categories and duplicated algebras}.  J. Algebra 305 (2006), no. 1, 548--561.}


\bibitem{ABST2}{I.~Assem, Th.~Br\"ustle, R.~Schiffler and G.~Todorov: \sl{m-cluster categories and m-replicated algebras}.  J. Pure Appl. Algebra 212 (2008), no. 4, 884-901.}


\bibitem{APT}{
I.~Assem, M.~I.~Platzeck,  and S.~Trepode: {\it
On the representation dimension of tilted and laura algebras}. J. Algebra 296 (2006), no. 2, 426--439. }

\bibitem{ARS}{M.~Auslander, I.~Reiten, and S.~O.~Smal\o:
\sl{Representation theory of artin algebras}. Cambridge Studies in
Advanced Mathematics 36 (1995).}

 
\bibitem{BIRS}{
A.~Buan, O.~Iyama,  I.~Reiten and D.~Smith: {\it
}. Amer. Journal of Math. 133 (2011) no. 4, 835-887. }



\bibitem{BMRRT}{
A.~Buan, R.~Marsh, M.~Reineke, I.~Reiten and G.~Todorov: {\it
Tilting theory and cluster combinatorics}. Adv. Math. 204
(2006), no. 2, 572--618.}

\bibitem{BMR}{A.~Buan, R.~Marsh, I.~Reiten: \sl{Cluster-tilted
      algebras}. Trans. Amer. Math. Soc., 239 (2007), no. 1, 323--332.}

\bibitem{BR}{A.~Buan,  I.~Reiten: \sl{From tilted to cluster-tilted
      algebras of Dynkin type}. Preprint (2005), arxiv:math.RT/0510445.}


\bibitem{FZ1}{S.~Fomin, A.~Zelevinsky: {\sl Cluster algebras I:
      Foundations}. J. Amer. Math. Soc. 15 (2002), no. 2, 497--529.}

\bibitem{FZ2}{S.~Fomin, A.~Zelevinsky: \sl{Cluster algebras II: Finite type
    classification}. Invent. Math. 154 (2003), no 1. 63--121.}

\bibitem{FGR}{R.~M.~Fossum, P.~A.~Griffith, I.~Reiten: \sl{Trivial extensions of abelian categories}. Lecture Notes in Mathematics 456 (1975), Springer Verlag.}

\bibitem{GHPRU}{S.~Gastaminza, D.~ Happel, M.~I.~Platzeck, M.~J.~Redondo, L.~Unger: {\sl Global dimensions for endomorphism algebras of tilting modules.}
Arch. Math. (Basel) 75 (2000), no. 4, 247--255.}

\bibitem{HRS}
D.~Happel,  I.~Reiten, S.~O.~Smal\o, {\sl Tilting in Abelian
Categories and Quasitilted Algebras}, Memoirs AMS {\bf 575}
(1996).
\bibitem{HU} {D.~Happel, L.~Unger: {\sl Almost complete tilting modules}. Proc. Amer. Math. Soc. 107 (1989), no. 3, 603--610.} 




\end{thebibliography}
\end{document}